\documentclass[12pt,a4paper]{amsart}
\usepackage{amsfonts}
\usepackage{amsthm}
\usepackage{amsmath}
\usepackage{amssymb}
\usepackage{amscd}
\usepackage{t1enc}
\usepackage[mathscr]{eucal}
\usepackage{indentfirst}
\usepackage{graphicx}
\usepackage{graphics}
\usepackage{pict2e}
\usepackage{epic}
\numberwithin{equation}{section}
\usepackage[margin=2.9cm]{geometry}
\usepackage{epstopdf} 
\usepackage{ulem}

\theoremstyle{plain}
\newtheorem{Thm}{Theorem}[section]
\newtheorem*{Thm*}{Theorem}
\newtheorem{Lem}[Thm]{Lemma}
\newtheorem{Cor}[Thm]{Corollary}
\newtheorem{Prop}[Thm]{Proposition}

 \theoremstyle{definition}
\newtheorem{Def}[Thm]{Definition}

\newtheorem{Rem}[Thm]{Remark}
\newtheorem{?}[Thm]{Problem}

\newcommand{\ovl}{\overline}
\newcommand{\p}{\partial}
\newcommand{\R}{\mathbb{R}}
\newcommand{\e}{\varepsilon}
\newcommand{\X}{\mathfrak{X}}
\newcommand{\de}{\triangleq}
\newcommand{\ml}{\min_{y \in \R} F_l(t,x,y)}
\newcommand{\mr}{\min_{y \in \R} F_r(t,x,y)}
\newcommand{\m}{\min_{y \in \R} F(t,x,y)}
\newcommand{\mln}{\min_{y\leq 0} F_l(t,x,y)}

\newcommand{\mrp}{\min_{y\geq 0} F_r(t,x,y)}

\begin{document}

%%%%%%%%%%%%%%%%%%%%%%%%%%%
% title,author,address, email
%%%%%%%%%%%%%%%%%%%%%%%%%%%

\title[Shock waves and rarefaction waves under periodic perturbations]{Asymptotic stability of shock waves and rarefaction waves under periodic perturbations for 1-D convex scalar conservation laws}

\author[Z. Xin]{Zhouping XIN}
\thanks{This research is partially supported by Zheng Ge Ru Foundation, Hong Kong RGC Earmarked Research Grants, CUHK-14300917, CUHK-14305315, and CUHK4048/13P, and NSFC/RGC Joint Research Grant N-CUHK 443-14. }
\address[Z. Xin]{The Institute of Mathematical Sciences \&  Department of Mathematics, The Chinese University of Hong Kong,	Shatin, N.T., Hong Kong}
\email{zpxin@ims.cuhk.edu.hk}

\author[Q. Yuan]{Qian YUAN}
\address[Q. Yuan]{The Institute of Mathematical Sciences \&  Department of Mathematics, The Chinese University of Hong Kong,	Shatin, N.T., Hong Kong}
\email{qyuan@math.cuhk.edu.hk}

\author[Y. Yuan]{Yuan YUAN}
\address[Y. Yuan]{South China Research Center for Applied Mathematics and Interdisciplinary Studies, South China Normal University,	Guangzhou, Guangdong, China}
\email{yyuan2102@m.scnu.edu.cn}

%\address{Massachusetts Institute of Technology \\ Department of Mathematics \\
%Cambridge MA 02139 \&  E\"{o}tv\"{o}s Lor\'{a}nd university \\ Department of Computer 
%Science \\ H-1117 Budapest
%\\ P\'{a}zm\'{a}ny P\'{e}ter s\'{e}t\'{a}ny 1/C \\ Hungary} 

%\email{peter.csikvari@gmail.com}


\begin{abstract} 
	In this paper we study large time behaviors toward shock waves and rarefaction waves under periodic perturbations for 1-D convex scalar conservation laws.
	The asymptotic stabilities and decay rates of shock waves and rarefaction waves under periodic perturbations are proved. 
	%Moreover, we give an optimal $L^{\infty}$ decay estimate for solutions with periodic initial data.
\end{abstract}

\maketitle

%%%%%%%%%%%%%%%%%%%%%%%%%%%%%%%%%%%%%%%%%%%
%% Introduction
%%%%%%%%%%%%%%%%%%%%%%%%%%%%%%%%%%%%%%%%%%%

\section{Introduction} 
We consider the Cauchy problem for convex scalar conservation laws in one-dimensional case,
\begin{equation}\label{equ1}
		\p_t u(x,t)+\p_x f(u(x,t)) =0, \quad x \in (-\infty, +\infty),\quad t>0,
\end{equation}
\begin{equation}\label{ic1}
		u|_{t=0}=
	\begin{cases}
		\ovl{u}_l+w_0(x)& \text{if~} x<0,\\
		\ovl{u}_r+w_0(x)& \text{if~} x>0,
	\end{cases}		
\end{equation}
where $f(u) \in C^2(\R) \text{ satisfies } f''(u)>0$, $\ovl{u}_l$ and $\ovl{u}_r$ are two distinct constants, $w_0(x) \in L^{\infty}(\R)$ is any periodic function with period $p>0$, and $\ovl{w}$ is its average
\begin{equation}\label{avg}
\ovl{w}\de \frac{1}{p} \int_0^p w_0(x) dx.
\end{equation}
When $w_0(x) \equiv 0$, the problem is Riemann problem, and its entropy solutions are shock waves if $\ovl{u}_l>\ovl{u}_r$ or rarefaction waves if $\ovl{u}_l<\ovl{u}_r$. In this paper, we plan to study the asymptotic stabilities of the solutions to \eqref{equ1} and \eqref{ic1} with bounded periodic perturbation $ w_0(x) $. 

The theory of convex scalar conservation laws is one of the most classical theory in PDE, and far-reaching results have been obtained.
As is well known, for any $L^{\infty}$ initial data, there exists a unique entropy solution to \eqref{equ1} in $Lip((0,+\infty),L^1_{loc}(\R))$ (\cite[Theorem~16.1]{smo}).

For the study of large time behaviors of entropy solutions to \eqref{equ1}, when initial data is in $L^\infty \cap L^1 $, the entropy solution decays to $0$ in $L^\infty$ norm at a rate $t^{-1}$. When initial data is bounded and has compact support, the entropy solution decays to the N-wave in $L^1$ norm at a rate $t^{-\frac{1}{2}}$, see \cite{hopf} and \cite{Lax}. 
While for the periodic initial data, which is obviously not in $L^1$, Glimm J. and P. Lax \cite[Theorem 5.2]{GD} seem to be the first to state that the entropy solutions decay to their average at a rate $t^{-1}$. 

It is well known that shock waves and rarefaction waves are two important and typical entropy solutions in genuinely nonlinear conservation laws
, and their stability problems are of great interest not only in mathematics but also in physics.
If the initial perturbation is compactly supported, Liu in \cite{Liu1978} proved that for shock waves, the perturbed solution becomes a translation of shock waves after a finite time; while for rarefaction waves, we only have that the perturbed solution converges to the centered rarefaction waves at a rate $t^{-\frac{1}{2}}$ in $L^{\infty}$ norm. 

However, when the perturbation remains oscillating at the infinity, like the periodic one, the stability of these simple waves are still open. Here the initial data \eqref{ic1} is neither integrable nor periodic on $\mathbb{R}$. 
%\red{ But fortunately, there are good properties in scalar conservation laws in one-dimensional case, like the generalized characteristics introduced by Dafermos in \cite{Dafe1}, and Hopf-Cole transformation in \cite{hopf} for Burger's equation. }

\vspace{0.5cm}

In this paper, we prove that for any given bounded periodic perturbation, the shock and rarefaction wave are both asymptotically stable. More precisely, we show that for shock waves, after a finite time, the perturbed shock consists of actually two periodic functions contacting with each other at a shock curve, and this shock curve tends to the background one at a rate $t^{-1}$, see Theorem \ref{thmshock}; while for perturbation of a rarefaction wave data, the solution consists of three parts separated by two distinct characteristics, where on two sides  the solution is periodic, and the perturbed rarefaction wave tends to the background one in $L^{\infty}$ norm at a rate $t^{-1}$, see Theorem \ref{thmrare}. 
The stability result for shock profiles under periodic perturbations for viscous case will be shown in a forthcoming paper. %delete it next time
Furthermore, we give a more exquisite convergent rate for the problem \eqref{equ1} with periodic initial data than Glimm and Lax \cite[Theorem 5.2]{GD}, see Theorem \ref{thmper},  and we also give a simple example \eqref{2constants} to show that this rate is optimal in some sense, see \eqref{lb1}. 

To prove the main results stated above, we make much use of some properties of generalized characteristics of convex scalar convex conservation laws (see Proposition \ref{propgen}, Lemma \ref{lemgre}) and the existence of divides of periodic solutions(see Lemma \ref{propdiper}),  which is a very special feature of periodic solutions and plays a essential role in our proof. Such concepts and tools were developed by Dafermos in \cite{Dafe1}, \cite{Dafe}.

At last, before the end of our paper, we present an alternative proof, inspired by Hopf-Cole transform in \cite{hopf}, to prove \eqref{glue} in Theorem \ref{thmshock} when $f(u)=u^2/2$.

\vspace{0.8cm}
%%%%%%%%%%%%%%%%%%%%%%%%%%%%%%%%%%%%%%%%%%%
%% Statement of main results
%%%%%%%%%%%%%%%%%%%%%%%%%%%%%%%%%%%%%%%%%%%
\section{Statement of main results} 
Before stating the main results of this paper, we firstly list the following result, which can be derived from \cite[Theorem 5.2]{GD} or \cite[Theorem 3.1]{Dafe1},
\begin{Thm*}
	Suppose that $u_0\in L^{\infty}$ is a periodic function of period $p$ with its average $\ovl{u}=\frac{1}{p}\int_0^p u_0(x)~dx$. Then for any $t>0$, the entropy solution $u(x,t)$ to \eqref{equ1} with initial data $u_0(x)$ is also  periodic of period $p$ with the same average $\ovl{u}$, and also
	\begin{equation}\label{lb13}
		|u(x,t)-\ovl{u}| \leq \frac{C}{t}, \quad \forall~ t>0,
	\end{equation}
	where $C$ depends on $ p, \ovl{u}, f$.
\end{Thm*}

\begin{Rem}
	For the asymptotic behavior of periodic solutions, after Glimm and Lax's result, Dafermos \cite[Theorem 3.1]{Dafe1} gave a more exquisite description of the asymptotic behavior, which behaves like a saw-toothed profile. In this paper we can give an optimal bound of $ \|u-\ovl{u}\|_{L^\infty} $ for the periodic solutions. This bound is more accurate than the results of Glimm and Lax, and it is optimal because it can be achieved for some special initial data.  Since this result is not related with the stability problems of shock and rarefaction waves,  we place the corresponding theorem and proof in the appendix.
\end{Rem}

In \eqref{ic1} we can assume that perturbation $w_0(x)$ has zero average
\begin{equation}\label{ave0}
\ovl{w}\de \frac{1}{p} \int_0^p w_0(x) dx =0
\end{equation} by replacing $\ovl{u}_l, \ovl{u}_r$ with $\ovl{u}_l+\ovl{w}, \ovl{u}_r+\ovl{w}$ respectively if necessary.

For $\ovl{u}_l>\ovl{u}_r$, the shock wave, $u^S$ is given by,
\begin{equation}\label{shock}
u^S(x,t)=\begin{cases}
\ovl{u}_l, & \text{~ if~ }x<st;\\
\ovl{u}_r, & \text{~ if~ }x>st.
\end{cases}
\quad \quad \text{~where } \quad s=\frac{f(\ovl{u}_l)-f(\ovl{u}_r)}{\ovl{u}_l-\ovl{u}_r}.
\end{equation}
and for $\ovl{u}_l<\ovl{u}_r$, the rarefaction wave, $u^R$ is,
\begin{equation}\label{rare}
u^R(x,t)\de 
\begin{cases}
\ovl{u}_l, &\text{~if~ } \frac{x}{t} < f'(\ovl{u}_l);\\
(f')^{-1}(\frac{x}{t}), &\text{~if~} f'(\ovl{u}_l)\leq \frac{x}{t}\leq f'(\ovl{u}_r);\\
\ovl{u}_r, &\text{~if~} \frac{x}{t}> f'(\ovl{u}_r);
\end{cases}
\end{equation}

In the rest of this paper, we will use the following notations to represent  different entropy solutions to problem \eqref{equ1} with different initial data, 
\begin{equation}\label{dsol}
	\begin{aligned}
	    w(x,t): &\text{~the entropy solution to \eqref{equ1} with } w(x,0)=w_0(x); \\
	    u_l(x,t): &\text{~the entropy solution to \eqref{equ1} with } u_l(x,0)=\ovl{u}_l+w_0(x);\\
	    u_r(x,t): &\text{~the entropy solution to \eqref{equ1} with } u_r(x,0)=\ovl{u}_r+w_0(x).
	\end{aligned}
\end{equation}

Then by \eqref{lb13}, one has
\begin{equation}\label{lb2}
|u_l(x,t)-\ovl{u}_l| \leq \frac{C}{t}, \quad |u_r(x,t)-\ovl{u}_r| \leq \frac{C}{t}, \quad \text{for~} \forall~ t>0, ~a.e.~x.
\end{equation}

Also we define the following extremal forward generalized characteristics (see definitions in Definition \ref{defgen} and Proposition \ref{propgen}) issuing from the origin $(0,0)$:
\begin{equation}\label{dX}
\begin{aligned}
X_-(t): \quad &\text{the minimal generalized characteristic associated with}~u\\
X_+(t): \quad &\text{the maximal generalized characteristic associated with}~u\\
X_{r-}(t): \quad &\text{the minimal generalized characteristic associated with}~u_r\\
X_{l+}(t): \quad &\text{the maximal generalized characteristic associated with}~u_l
\end{aligned}
\end{equation}
\vspace{0.5cm}

Then the main results of this paper are stated as follows:

\begin{Thm} \label{thmshock} 
	Suppose that $\ovl{u}_l>\ovl{u}_r$.
	Then for any periodic perturbation $w_0(x) \in L^{\infty}(\R) $ satisfying \eqref{ave0}, there exist a finite time $T_S>0$, and a unique curve  $X(t) \in \text{Lip}~(T_S,+\infty)$, which is actually a shock, such that for any $t>T_S$,
	\begin{equation} \label{glue}
	%		& u(X(t)-0,t)=u_l(X(t)-0,t),~ u(X(t)+0,t)=u_r(X(t)+0,t),\label{sho}\\	
	u(x,t)=\begin{cases}
	u_l(x,t),\quad  \text{if~} x<X(t),\\
	u_r(x,t),\quad \text{if~} x>X(t).
	\end{cases}
	\end{equation} 
	Moreover, 
	\begin{equation}\label{lb}
	\sup_{x<X(t)} |u(x,t)-\ovl{u}_l| +\sup_{x>X(t)} |u(x,t)-\ovl{u}_r| + |X(t)-st| \leq \frac{C}{t}, ~\forall~ t>T_S,
	\end{equation}
	Here $C$ and $T_S$ depend on $p, \ovl{u}_l, \ovl{u}_r, f$.
\end{Thm}

\begin{Thm}\label{thmrare}
	Suppose that $\ovl{u}_l<\ovl{u}_r$.
	Then for any periodic perturbation $w_0(x) \in L^{\infty}(\R)$
    satisfying \eqref{ave0},
    and for any $t>0$, 
	\begin{equation}
	|u(x,t)-u^R(x,t)| \leq \frac{C}{t}, \quad a.e. ~ x\in \R.
	\end{equation}
	where $C$ depends on $p, \ovl{u}_l, \ovl{u}_r, f$.
\end{Thm}

\begin{Thm}\label{thmshock2}
		Suppose that the assumptions of Theorem \ref{thmshock} hold, and  additionally $f(u)=u^2/2$, i.e., \eqref{equ1} is the Burger's equation. Then $$\text{when~ } t>T_S \text{~and~ } \dfrac{(\ovl{u}_l-\ovl{u}_r)t}{p} \text{~is an integer},\quad X(t)=st.$$
\end{Thm}

\begin{Thm}\label{thmrare2}
	Suppose that the assumptions of Theorem \ref{thmrare} hold, and  additionally $w_0$ satisfies 
    \begin{equation*}
    \int_0^x w_0(y)~dy \geq 0, \quad 0 \leq x\leq p,
    \end{equation*}
    then \begin{equation*}
	u(x,t)=\begin{cases}
	u_l(x,t),   & \text{if~} \frac{x}{t}<f'(\ovl{u}_l),\\
	(f')^{-1}(\frac{x}{t})=u^R(x,t), & \text{if~} f'(\ovl{u}_l)\leq \frac{x}{t}\leq f'(\ovl{u}_r);\\
	u_r(x,t),   & \text{if~} \frac{x}{t}>f'(\ovl{u}_r).
	\end{cases}
    \end{equation*}
\end{Thm}

\vspace{0.5cm}

This paper proceeds as follows: In Section 3, we present some well-known results on generalized characteristics, especially the divides, which can be found in Dafermos's book \cite{Dafe}, and we also obtain some propositions that will be frequently used; Theorem \ref{thmshock}-\ref{thmrare2} are proved in Section 4 and 5; in Section 6, for the special case $f(u)=\frac{u^2}{2}, $ i.e. Burger's equation, we give another proof inspired by the Hopf-Cole transform, to prove Theorem \ref{thmshock}; and Theorem \ref{thmper} and its proof are shown in Appendix A.

\vspace{0.8cm}
%%%%%%%%%%%%%%%%%%%%%%%%%%%%%%%%%%%%%%%%%%%
%% preliminary
%%%%%%%%%%%%%%%%%%%%%%%%%%%%%%%%%%%%%%%%%%%
\section{Preliminary: generalized characteristics}
Here we list some well-known results on generalized characteristics, which can be found in Chapter 10 and Chapter 11 in \cite{Dafe}.

\begin{Def}\label{defgen}
	A generalized characteristic for \eqref{equ1}, associated with the entropy solution $u(x,t)$, on the time interval $[\sigma, \tau] \subset [0,+\infty)$, is a Lipschitz function \\ $\xi: [\sigma, \tau] \longrightarrow (-\infty, +\infty)$ which satisfies the differential inclusion 
	\begin{equation*}
		\xi'(t) \in [f'(u(\xi(t)+,t)),f'(u(\xi(t)-,t))], \quad \text{ a.e. on} \quad [\sigma, \tau]
	\end{equation*}
\end{Def}

\begin{Prop}\label{propgen}
	Assume $u(x,t) $ is the entropy solution to \eqref{equ1} with $L^{\infty} $ intial data $u_0$, then through any point $(\ovl{x},\ovl{t}) \in (-\infty,+\infty)\times [0,+\infty) $ pass two extremal generalized characteristics(which may not be distinct) defined on $[0,+\infty)$, namely the minimal $\xi_-(t)$ and the maximal $\xi_+(t)$ with $\xi_-(t)\leq \xi_+(t)$ for $t \in [0,+\infty)  $.
	And for any generalized characteristic $\xi(t)$ passing through $(\ovl{x},\ovl{t})$, there holds $\xi_-(t) \leq \xi(t) \leq \xi_+(t), \forall~t\geq 0$.

	Furthermore, if $\ovl{t}>0, $ then the minimal backward (confined in $0\leq t\leq \ovl{t}$) characteristic $\xi_-(t)$ and maximal backward characteristic $\xi_+(t)$ are both straight lines, and satisfies for $0<t<\ovl{t}$,
	\begin{equation}\label{genpro}
		\begin{aligned}
			& u_0(\xi_-(0)-)\leq u(\xi_-(t)-,t)=u(\xi_-(t)+,t)=u(\ovl{x}-,\ovl{t}) \leq u_0(\xi_-(0)+);\\
			& u_0(\xi_+(0)-)\leq u(\xi_+(t)-,t)=u(\xi_+(t)+,t)=u(\ovl{x}+,\ovl{t})\leq u_0(\xi_+(0)+);			
		\end{aligned}
	\end{equation}
	and the forward (confined in $t\geq \ovl{t}$) characteristic is unique, i.e. for $t \geq \ovl{t},$
	\begin{equation}\label{foruni}
		\xi_-(t)=\xi_+(t) \de \xi(t).
	\end{equation}
	See Figure \ref{gc1} and Figure \ref{gc2}.
\end{Prop}

%Figures
%%%%%%%%%%%%%%%%%%%%%%%%%%%%%%%%%%%%%%%%%%%%%
\begin{figure}[h]
	\includegraphics[scale=0.5]{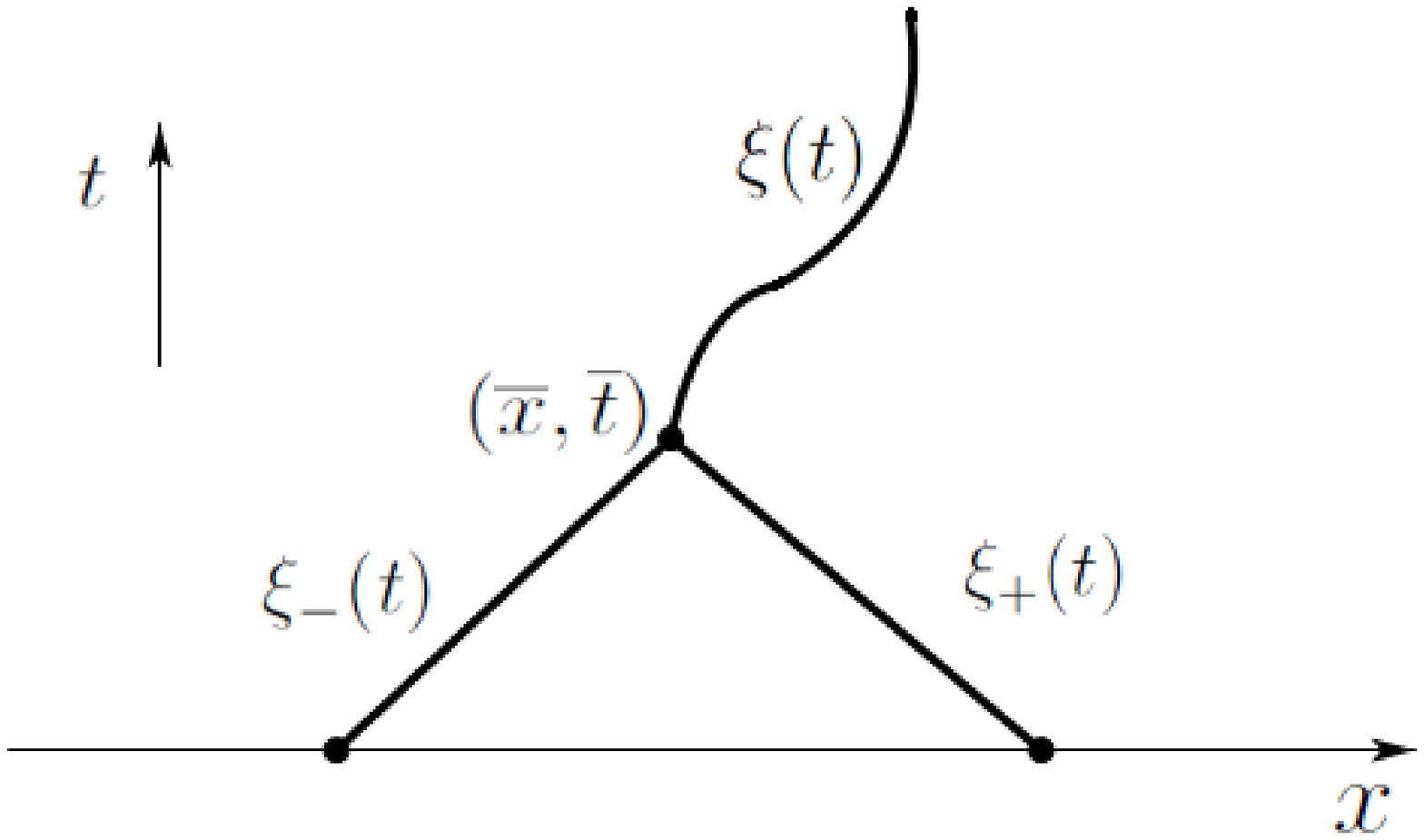} 
	\caption{}
	\label{gc1}
\end{figure}

\begin{figure}[h]
	\includegraphics[scale=0.5]{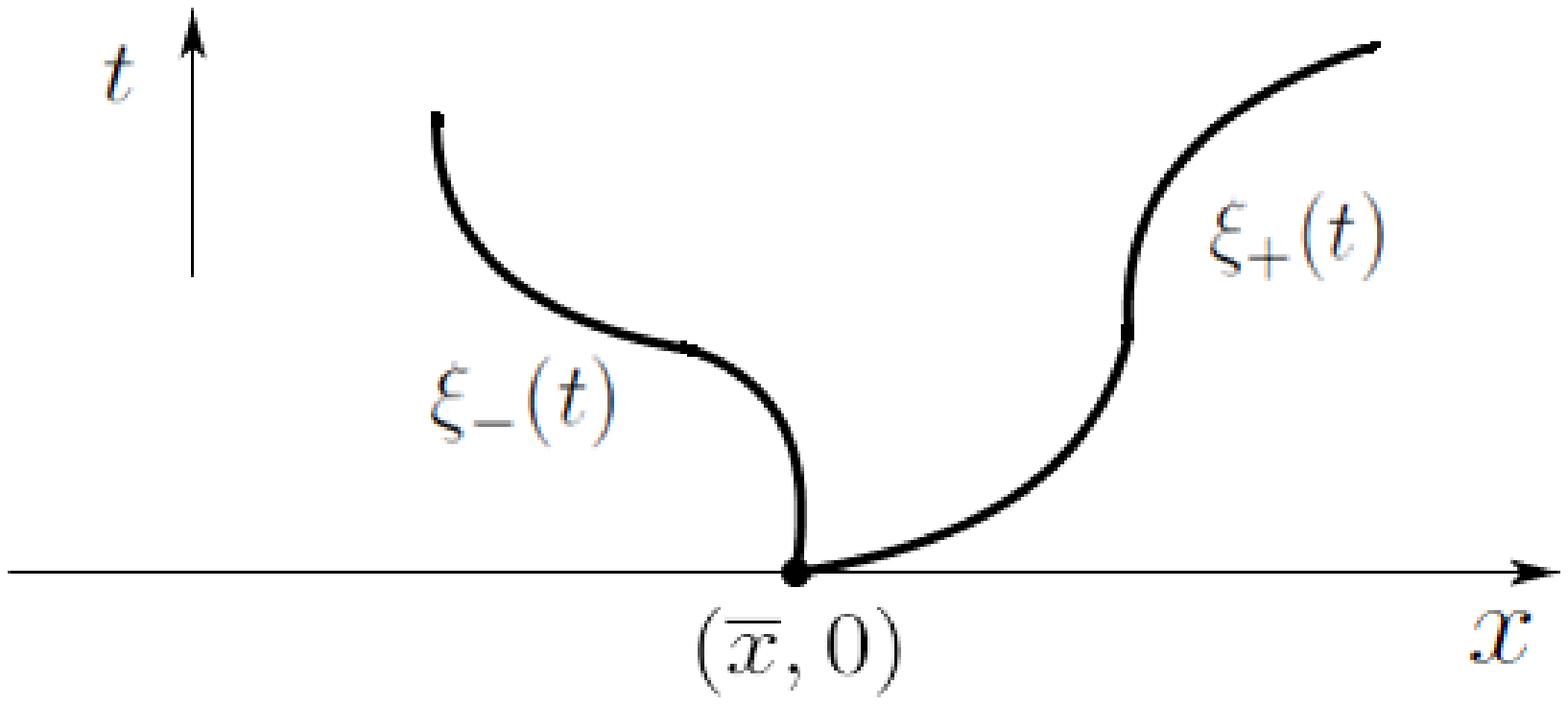} 
	\caption{}
	\label{gc2}
\end{figure}
%%%%%%%%%%%%%%%%%%%%%%%%%%%%%%%%%%%%%%%%%%%%%%

\begin{Rem}
\begin{enumerate}
	\item For $\ovl{t}>0, $ the minimal backward characteristic $\xi_-(t)$ and the maximal backward characteristic $\xi_+(t)$ coincide if and only if $u(\ovl{x}-,\ovl{t})=u(\ovl{x}+,\ovl{t})$. 
    \item For any two extremal forward generalized characteristic $\xi_-(t)$ and $\xi_+(t)$ issuing from $x-$axis, if they coincide at some time $t_0>0,$ then they remain the same for all $t>t_0. $ 
\end{enumerate}
\end{Rem}

The following useful integral formula, \eqref{tri}, can be found in \cite{Dafe}.

\begin{Lem}\label{lemgre}
	Let $\xi(t)$ and $\widetilde{\xi}(t)$ be two extremal backward characteristics corresponding to entropy solutions $u(x,t)$ and $\widetilde{u}(x,t)$ to \eqref{equ1} with $L^{\infty}$ initial data $u(x,0)$ and $ \widetilde{u}(x,0) $ respectively,  emanating from a fixed point $(\ovl{x},\ovl{t}) \in (-\infty,+\infty)\times (0,+\infty)$, see Figure \ref{triFi}. Then if $ ~\widetilde{\xi}(0) <\xi(0) $, it holds that
	\begin{align}
	& \int_0^{\ovl{t}} \{ f(b)-f(\widetilde{u}(\xi(t)-,t)) - f'(b)[b- \widetilde{u}(\xi(t)-,t)] \} ~dt \notag \\
	& + \int_0^{\ovl{t}} \{ f(\widetilde{b})-f(u(\widetilde{\xi}(t)+,t)) - f'(\widetilde{b})[\widetilde{b}-u(\widetilde{\xi}(t)+,t)] \} ~dt \label{tri}\\
	& \quad \quad \quad \quad \quad =\int_{\widetilde{\xi}(0)}^{\xi(0)} [u(x,0)-\widetilde{u}(x,0)]~dx. \notag
	\end{align}
	where $b $ and $\widetilde{b}$ are constant defined by
	\begin{align*}
	b \de
	\begin{cases}
	u(\ovl{x}-,\ovl{t}), \quad \text{if~ } \xi(t) ~\text{is minimal;}\\
	u(\ovl{x}+,\ovl{t}), \quad \text{if~ } \xi(t) ~\text{is maximal.}
	\end{cases}\\
	\widetilde{b} \de
	\begin{cases}
	\widetilde{u}(\ovl{x}-,\ovl{t}), \quad \text{if~ } \widetilde{\xi}(t) ~\text{is minimal;}\\
	\widetilde{u}(\ovl{x}+,\ovl{t}), \quad \text{if~ } \widetilde{\xi}(t) ~\text{is maximal.}
	\end{cases}
	\end{align*}
\end{Lem}
%%%%%%%%%%%%%%%%%%%%%%%%%%%%%%%%%%%%%%%%%%%%%
\begin{figure}[h]
	\includegraphics[scale=0.5]{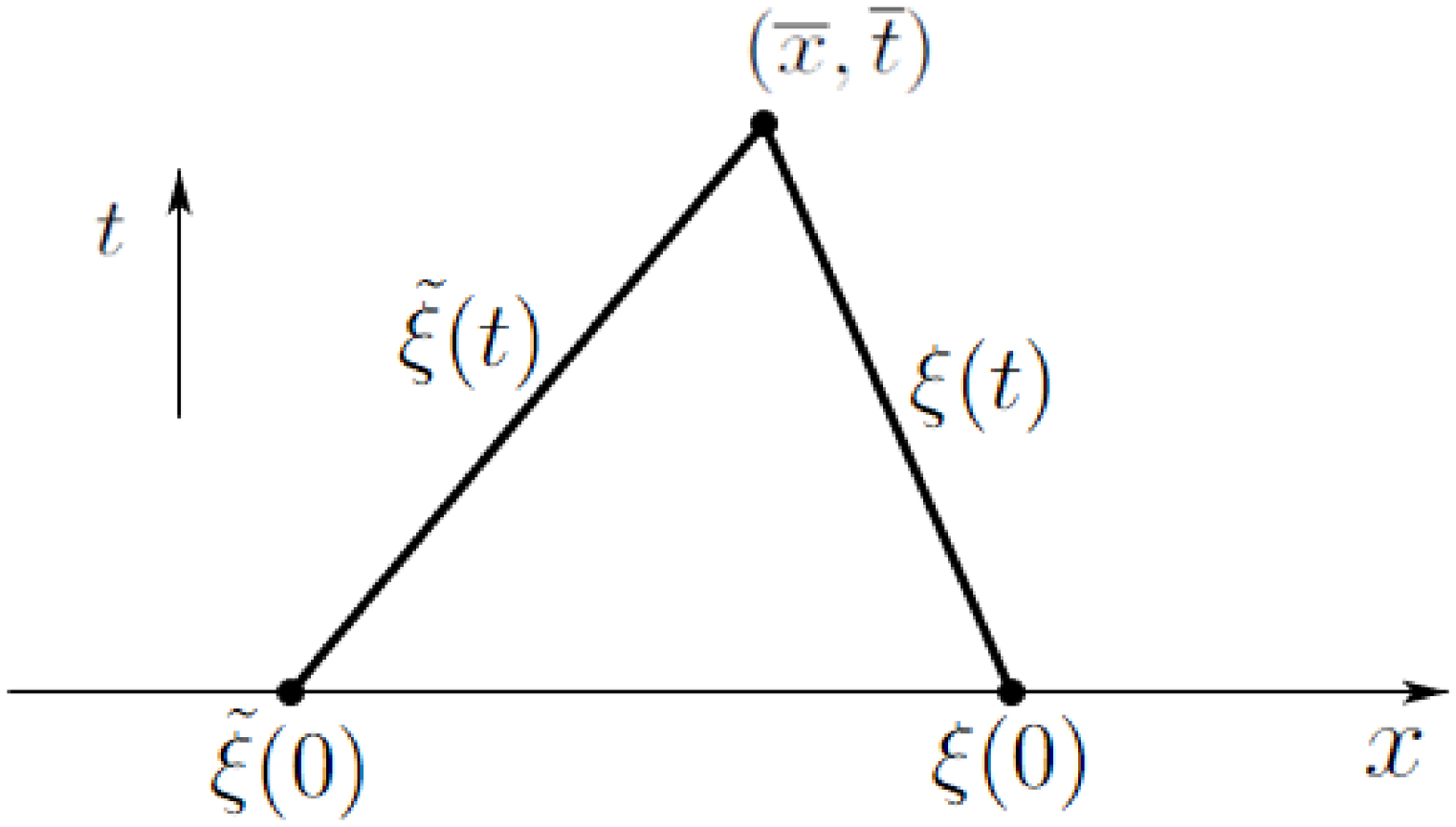} 
	\caption{}
	\label{triFi}
\end{figure}
%%%%%%%%%%%%%%%%%%%%%%%%%%%%%%%%%%%%%%%%%%%%%
\begin{proof}[Proof of Lemma \ref{lemgre}]
	By Proposition \ref{propgen}, it holds that $u(\xi(t),t)\equiv b $ and $\widetilde{u}(\widetilde{\xi}(t),t) \equiv \widetilde{b}$, for $0<t<\ovl{t}$. Then integrating the equation $$ \p_t (u-\widetilde{u}) +\p_x(f(u)-f(\widetilde{u}))=0 $$ in the triangle with vertex $(\ovl{x},\ovl{t}), (\widetilde{\xi}(0),0), (\xi(0),0)$, and using Green's formula, one can get easily \eqref{tri}, for details, see \cite{Dafe}.
\end{proof}

\begin{Def}[ \cite{Dafe} Definition~10.3.3]\label{defdivi}
	A minimal (or maximal) divide, associated with the solution u, is a Lipschitz function $ \xi(t): [0, +\infty] \rightarrow \R $ such that $ \xi(t) = \lim_{m\rightarrow \infty} \xi_m(t), $ uniformly on compact time intervals, where $ \xi_m(\cdot) $ is the minimal (or maximal) backward characteristic emanating from a point $ (x_m, t_m), $ with $ t_m \rightarrow +\infty, $ as $ m \rightarrow \infty. $ 
\end{Def}

\begin{Prop}[ \cite{Dafe} Theorem~11.4.1]\label{propdivi}
	Considering the Cauchy problem for \eqref{equ1} with any $L^{\infty}$ initial data $u_0(x). $ If there exists a constant $\ovl{u} $ and a point $\ovl{x} \in \R $, s.t. 
	\begin{equation}\label{divi}
		\int_{\ovl{x}}^x [u_0(y)- \ovl{u}]~dy \geq 0, \quad -\infty <x< \infty,
	\end{equation}
	then there exists a divide associated with $ u $, issuing from the point $(\ovl{x},0)$ of the $x$-axis, on which the entropy solution $u $ is constant $\ovl{u}.$
\end{Prop}

\vspace{0.5cm}
In the following of this section, we give some conclusions derived from the knowledge above, which will be frequently used in this paper.

\begin{Prop}\label{propdiper}	
	Assume that the initial data $u_0(x) \in L^{\infty}$ is periodic of period $p$ with the average $\ovl{u}=\frac{1}{p}\int_0^p u_0(x)~dx$. Then for each integer $N \in \mathbb{Z}$, the straight line
	\begin{equation}\label{diper}
		x=f'(\ovl{u})t+\ovl{x}+Np
	\end{equation}
	is a divide associated with the entropy solution $u(x,t)$ to \eqref{equ1}.\\
	Here $\ovl{x} $ is defined as some point in $[0,p)$, satisfying
	\begin{equation}\label{minpoint}
		\int_0^{\ovl{x}} [u_0(y)-\ovl{u}] dy = \min_{x \in [0,p]} \int_0^x [u_0(y)-\ovl{u}] dy.
    \end{equation}
\end{Prop}

\begin{proof}[Proof of Proposition \ref{propdiper}]
	The function $ \int_0^x [u_0(y)-\ovl{u}]dy $ is continuous and it's easy to verify that this integral is periodic with period $p$ due to the conservation form of the equation \eqref{equ1}. Then there exists a point $\ovl{x} \in [0,p), $ s.t. \eqref{minpoint} holds. And combining with $ \int_0^p [u_0(y)-\ovl{u}]dy =0 $, it's easy to verify that $\int_{\ovl{x}}^x [u_0(y)-\ovl{u}]~dy \geq 0, ~ -\infty <x< \infty.$
	
	Then  for any $N \in \mathbb{Z}, $
	\begin{equation}\label{geq0}
		\int_{\ovl{x}+Np}^x [u_0(y)-\ovl{u}] ~dy=\int_{\ovl{x}}^x [u_0(y)-\ovl{u}] ~dy \geq 0, \quad -\infty <x< \infty.
	\end{equation}
	So by Proposition \ref{propdivi} and \eqref{geq0}, \eqref{diper} is a divide for $u(x,t)$.
\end{proof}

For the periodic perturbation $ w_0(x), $ where $w_0$ satisfies \eqref{ave0}, one can choose a point $ a \in [0,p), $ such that 
\begin{equation}\label{defa}
	\int_0^a [w_0(y)-\ovl{w}] dy = \min_{x \in [0,p]} \int_0^x [w_0(y)-\ovl{w}] dy.
\end{equation}

Then by Proposition \ref{propdiper} and \eqref{defa}, it's easy to verify
\begin{Cor}\label{corgamma}
	For the entropy solution $ u_l(x,t)~ (resp.~ u_r(x,t))$ to \eqref{equ1} with initial data $u_l(x,0)=\ovl{u}_l+w_0(x)~ (resp. ~u_r(x,0)=\ovl{u}_r+w_0(x))$, and each $N \in \mathbb{Z}, $ the straight lines
	\begin{equation}\label{dilr}
		 x=\Gamma_l^N(t)\de a+Np+f'(\ovl{u}_l)t,~\Big( resp. ~x=\Gamma_r^N(t)\de a+Np+f'(\ovl{u}_r)t \Big) 
	\end{equation}
	are divides associated with $u_l$ (resp. $u_r$) on which $u_l(x,t)\equiv \ovl{u}_l~ ( resp. ~u_r(x,t)\equiv \ovl{u}_r ).$ 
\end{Cor}

By Lemma \ref{lemgre}, one can prove 
\begin{Lem}\label{lemglue}
	Let $c_1>c_2$ be two constants and $u_0(x)\in L^{\infty}(\R)$, and let $u_1(x,t), u_2(x,t), u_{12}(x,t)$ be the entropy solutions to \eqref{equ1} with their corresponding initial data
	\begin{align*}
		& u_1(x,t=0)=c_1+u_0(x),\\
		& u_2(x,t=0)=c_2+u_0(x),\\
		& u_{12}(x,t=0)=
		\begin{cases}
			c_1+u_0(x), \quad \text{if ~} x<0,\\
			c_2+u_0(x), \quad \text{if ~} x>0.
		\end{cases}
	\end{align*}
	Let $x_-(t), x_+(t)$ be the minimal and maximal forward generalized characteristics issuing from the origin associated with $u_{12}$ (see Figure \ref{ufig}, note that $x_-$ and $x_+$ may not be distinct), then 
	\begin{equation}
		u_{12}(x,t) =
		\begin{cases}
			u_1(x,t), \quad \text{if ~} x<x_+(t),\\
			u_2(x,t), \quad \text{if ~} x>x_-(t).
		\end{cases}
	\end{equation}
\end{Lem}

\begin{proof}[Proof of Lemma \ref{lemglue}]
	Without loss of generality, we prove only the case when $x<x_+(t)$.
	
	For any fixed $(\ovl{x},\ovl{t}) $ with $\ovl{x}<x_+(\ovl{t}),~ \ovl{t}>0$, we firstly prove that $u_{12}(\ovl{x}+,\ovl{t})=u_1(\ovl{x}+,\ovl{t})$.
	Through $(\ovl{x},\ovl{t}) $ we draw the maximal backward characteristics $\xi_+(t)$ and $\eta_+(t)$ corresponding to the entropy solutions $u_{12}$ and $u_1$ respectively. By Proposition \ref{propgen}, $\xi_+(t)$ and $\eta_+(t)$ are both straight lines, and for $0<t<\ovl{t}$, 
	$$ u_{12}(\xi_+(t)+,t)=u_{12}(\xi_+(t)-,t)=u_{12}(\ovl{x}+,\ovl{t}),\quad \xi_+'(t)=f'(u(\ovl{x}+,\ovl{t})) $$
	$$u_1(\eta_+(t)+,t)=u_1(\eta_+(t)-,t)=u_1(\ovl{x}+,\ovl{t}),\quad \eta_+'(t)=f'(u_1(\ovl{x}+,\ovl{t})) $$
	Then $\xi_+(0)\leq 0$ since $\xi_+(t)$ cannot cross through another generalized characteristic $x_+(t)$ at $t>0$ (since the forward characteristic issuing from any point $(x,t)$ with $t>0$ is unique, by Proposition \ref{propgen}). See Figure \ref{ufig}.

	%%%%%%%%%%%%%%%%%%%%%%%%%%%%%%%%%%%%%%%%%%%%%
	\begin{figure}[h]
		\includegraphics[scale=0.4]{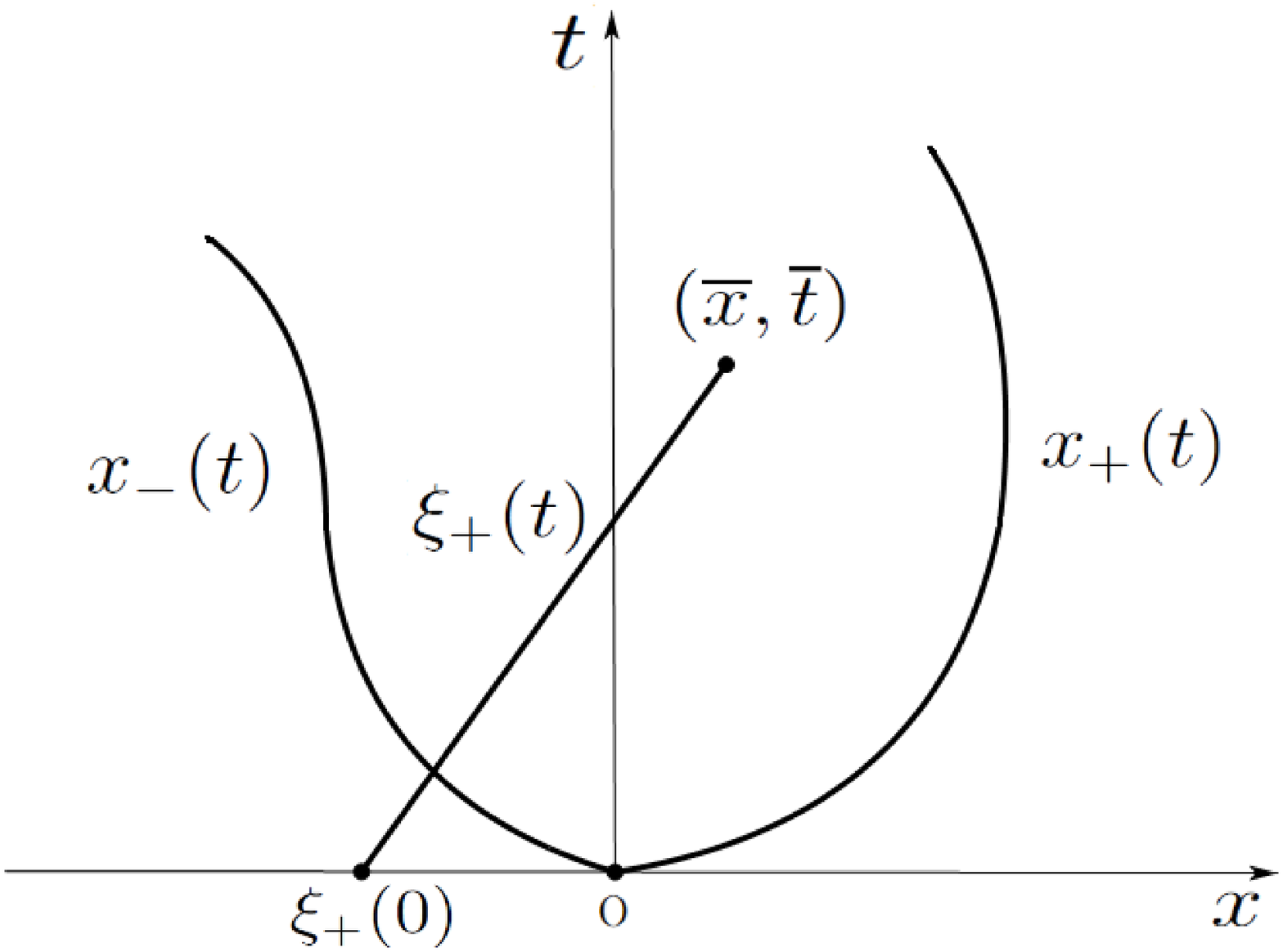} 
		\caption{}
		\label{ufig}
	\end{figure}
	%%%%%%%%%%%%%%%%%%%%%%%%%%%%%%%%%%%%%%%%%%%%%
	
	If $u_{12}(\ovl{x}+,\ovl{t})> u_1(\ovl{x}+,\ovl{t})$, then $\xi_+(0)<\eta_+(0)$. using \eqref{tri} with $u=u_1, \widetilde{u}=u_{12}, \xi=\eta_+, \widetilde{\xi}=\xi_+$, one can have that 
	\begin{align}
	& \int_0^{\ovl{t}} \{ f(b)-f(u_{12}(\eta_+(t)-,t)) - f'(b)[b- u_{12}(\eta_+(t)-,t)] \} ~dt \notag \\
	& + \int_0^{\ovl{t}} \{ f(\widetilde{b})-f(u_1(\xi_+(t)+,t)) - f'(\widetilde{b})[\widetilde{b}-u_1(\xi_+(t)+,t)] \} ~dt \label{tril} \\
	& =\int_{\xi_+(0)}^{\eta_+(0)} (u_1(x,0)-u_{12}(x,0)) ~dx \notag \\
	& =\begin{cases}
	0,\quad \quad \quad \quad \quad \quad \quad\quad\quad\quad \text{~if ~} \eta_+(0)\leq 0;\\
	\int_0^{\eta_+(0)} (c_1-c_2) ~dx >0, \quad\text{~if ~} \eta_+(0)> 0
	\end{cases} \quad\quad\quad \geq 0.	\notag
	\end{align}
	here $b=u_1(\ovl{x}+,\ovl{t}), \widetilde{b}=u_{12}(\ovl{x}+,\ovl{t}).$ 

	By the strict convexity of $f$, the left hand side of \eqref{tril} is non-positive, and we have that for $t \in [0,\ovl{t}]$,
	\begin{align}
	& u_{12}(\eta_+(t)-,t) \equiv b= u_1(\ovl{x}+,\ovl{t}),\label{for1} \\
	& u_1(\xi_+(t)+,t) \equiv \widetilde{b}=u_{12}(\ovl{x}+,\ovl{t}).\label{for2}
	\end{align}
	Then \eqref{for1} implies that 
	$$ \eta_+'(t)=f'(u_1(\ovl{x}+,\ovl{t}))=f'(u_{12}(\eta_+(t)-,t))$$
	which means that $\eta_+(t)$ is a backward generalized characteristic through $(\ovl{x},\ovl{t})$ associated with $u$. However, $\xi_+(t)$ is the maximal backward characteristic associated with $u$, thus there must hold $\eta_+(t) \leq \xi_+(t), t\in [0,\ovl{t}]$, which contradicts with $\xi_+(0)<\eta_+(0)$.
	
	Similarly, if $ u_{12}(\ovl{x}+,\ovl{t})< u_1(\ovl{x}+,\ovl{t}) $, it means $ \eta_+(0)<\xi_+(0)\leq 0 $. Then same argument as above can verify that for $ t \in [0,\ovl{t}], u_1(\xi_+(t)-,t) \equiv u_{12}(\ovl{x}+,\ovl{t}) $, which implies that $ \xi_+'(t)=f'(u_1(\xi_+(t)-,t)) $. It means that $\xi_+(t)$ is a backward generalized characteristic associated with $u_1$, then there holds $ \xi_+(t)\leq \eta_+(t), t\in [0,\ovl{t}] $, which is also a contradiction. 
	
	And the proof to $ u_{12}(\ovl{x}-,\ovl{t})=u_1(\ovl{x}-,\ovl{t}) $ is similar. Since one can draw the minimal backward characteristics $\xi_-(t)$ and $\eta_-(t)$ corresponding to $u_{12}$ and $u_1$ respectively, and then use the similar argument as above.
\end{proof}

By Lemma \ref{lemglue}, one can easily prove

\begin{Prop}\label{propglue} The following properties hold:
	\begin{enumerate}
		\item[(i).] If $\ovl{u}_l>\ovl{u}_r$, then
		\begin{equation*}
			u(x,t)=
			\begin{cases}
			    u_l(x,t), &\text{~if~} x<X_+(t),\\
			    u_r(x,t), &\text{~if~} x>X_-(t).
			\end{cases}
		\end{equation*}

		\item[(ii).] If $\ovl{u}_l<\ovl{u}_r$, then
		\begin{equation*}
			u(x,t)=
			\begin{cases}
			    u_l(x,t), &\text{~if~} x<X_{l+}(t),\\
			    u_r(x,t), &\text{~if~} x>X_{r-}(t).
			\end{cases}
		\end{equation*}
	\end{enumerate}
	Here $X_{l+}, X_{r-}, X_{\pm} $ are defined in \eqref{dX}.
\end{Prop}
\begin{proof}[Proof of Proposition \ref{propglue}] $ $\\
	\begin{enumerate}
		\item[(i).] Note that $\ovl{u}_l>\ovl{u}_r$. Thus one needs to take $ u_{12}=u, u_1=u_l, u_2=u_r$ in Lemma \ref{lemglue}, and then (1) follows easily.\\
		
		\item[(ii).] Note that $\ovl{u}_l<\ovl{u}_r$ and
		\begin{equation*}
		\begin{aligned}
		& u_l(x,t=0)=\begin{cases}
		u(x,t=0) & \text{if}~ x<0,\\
		\ovl{u}_l-\ovl{u}_r+u(x,t=0) & \text{if}~ x>0,\\
		\end{cases}\\
		& u_r(x,t=0)=\begin{cases}
		\ovl{u}_r-\ovl{u}_l+u(x,t=0) & \text{if}~ x<0,\\
		u(x,t=0) & \text{if}~ x>0.\\
		\end{cases}
		\end{aligned}
		\end{equation*} 
		Therefore, by taking $u_{12}=u_l, u_1=u$ and $u_{12}=u_r, u_2=u$ respectively in Lemma \ref{lemglue}, one can prove (2).
	\end{enumerate}
\end{proof}

\vspace{0.8cm}

%%%%%%%%%%%%%%%%%%%%%%%%%%%%%%%%%%%%%%%%%%%
%% shock waves
%%%%%%%%%%%%%%%%%%%%%%%%%%%%%%%%%%%%%%%%%%%

\section{Proof of Theorem \ref{thmshock} and \ref{thmshock2}}
In this section we prove Theorem \ref{thmshock} and \ref{thmshock2}. 

\begin{proof}[Proof of Theorem \ref{thmshock}]
By Proposition \ref{propglue}, if $X_-(t) \equiv X_+(t)$ for $t \in [0,+\infty)$, then \eqref{glue} holds immediately.  

If $X_-(t)<X_+(t)$, then $u_l(x,t)$ coincides with $u_r(x,t)$ for $X_-(t)<x<X_+(t)$. But by \eqref{lb2}, after a finite time $T_S$, say $\dfrac{2C}{T_S} = \ovl{u}_l-\ovl{u}_r $, here $ C $ is the constant in \eqref{lb2}, it holds that
$$ u_l(x,t)>u_r(x,t), \quad  \text{~for~}  ~t>T_S,~ -\infty<x<+\infty, $$
which means that for $t>T_S$, $X_-(t) \equiv X_+(t)$ must hold. Thus \eqref{glue} is proved, then it remains to prove \eqref{lb}.

\vspace{0.5cm}

	For $N>0$, we define the trapezium:
	\begin{equation}\label{tra}
	\Omega^N(T) \de \{(x,t): 0\leq t\leq T,~ \Gamma_l^{-N}(t) \leq x \leq \Gamma_r^N(t)~\}
	\end{equation}
	For each $T>T_S, $ one can choose $N>0 $ large enough, s.t.
	\begin{equation}\label{Ga}
	\Gamma_l^{-N}(t) <X_-(t), \quad \Gamma_r^N(t) >X_+(t), \quad \text{for all~} 0\leq t\leq T, 
	\end{equation}
	see Figure \ref{green}.\\		
	% input a figure
	%%%%%%%%%%%%%%%%%%%%%%%%%%%
	\begin{figure}[h]
		\includegraphics[scale=0.35]{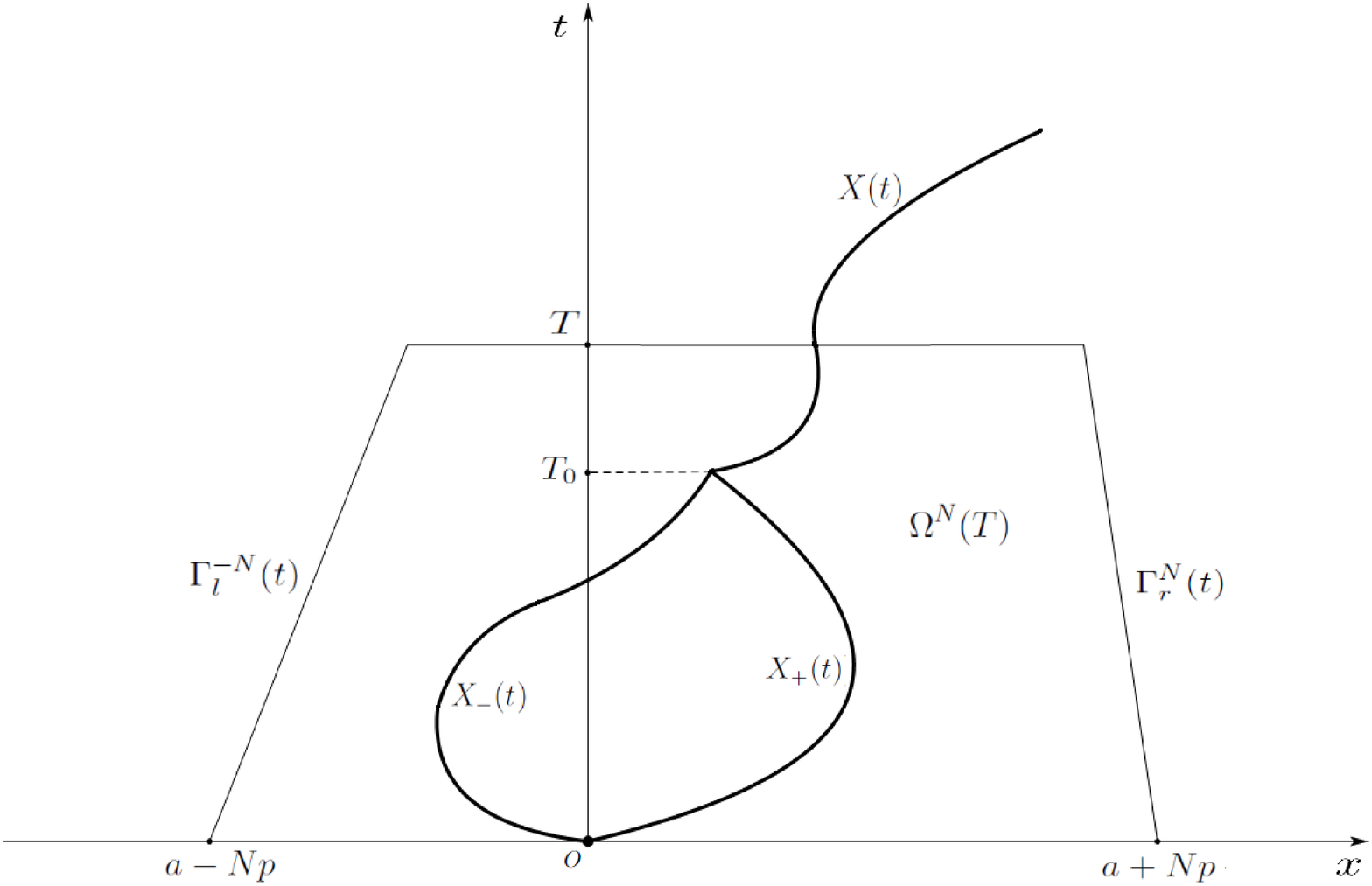} 
		\caption{}
		\label{green}
	\end{figure}
	%%%%%%%%%%%%%%%%%%%%%%%%%%% 

	Then applying the Green Formula in $\Omega^N(T) $ yields that
	\begin{align*}
	0=& \int_{\Omega^N(T)} \Big( \p_t u +\p_x f(u) \Big)~dxdt = \int_{\Gamma_l^{-N}(T)}^{\Gamma_r^N(T)} u(x,T) ~dx \\
	& -\int_{\Gamma_l^{-N}(0)}^{\Gamma_r^N(0)} u(x,0)~dx - \int_0^T \Big\{ f'(\ovl{u}_r)~u_r(\Gamma_r^N(t),t)-f\Big(u_r(\Gamma_r^N(t),t)\Big) \Big\}~dt \\
	& +\int_0^T \Big\{ f'(\ovl{u}_l)~u_l(\Gamma_l^{-N}(t),t)-f\Big(u_l(\Gamma_l^{-N}(t),t)\Big) \Big\}~dt \de I_1+I_2+I_3+I_4.
	\end{align*}
	It follows from \eqref{glue} that
	\begin{align}
	I_1= & \int_{\Gamma_l^{-N}(T)}^{X(T)} u_l(x,T)~dx +\int_{X(T)}^{\Gamma_r^N(T)} u_r(x,T)~dx \label{I1} \\
	= & \int_{\Gamma_l^{-N}(T)}^{X(T)} \Big(u_l(x,T)-\ovl{u}_l\Big)~dx +\int_{X(T)}^{\Gamma_r^N(T)} \Big(u_r(x,T)-\ovl{u}_r \Big)~dx \notag\\
	& +(X(T)-\Gamma_l^{-N}(T))\ovl{u}_l+ (\Gamma_r^N(T)-X(T))\ovl{u}_r \notag
	\end{align}
	Note that 
	\begin{equation*}
		\int_x^{x+p} \Big(u_l(y,t)-\ovl{u}_l\Big)~dy=\int_x^{x+p} \Big(u_r(y,t)-\ovl{u}_r\Big)~dy=0, \quad \forall~ x \in \R.
	\end{equation*}
	So by \eqref{lb2}, the first two terms in $I_1$ satisfy
	\begin{equation}\label{leq}
	\Big|\int_{\Gamma_l^{-N}(T)}^{X(T)} \Big(u_l(x,T)-\ovl{u}_l\Big)~dx +\int_{X(T)}^{\Gamma_r^N(T)} \Big(u_r(x,T)-\ovl{u}_r \Big)~dx \Big|\leq \frac{C}{T}.
	\end{equation}
	
	By \eqref{ave0}, $\Gamma_l^{-N}(0)=a-Np$ and $\Gamma_r^N(0)=a+Np, $ one has 
	\begin{align}
	I_2=& -\int_{\Gamma_l^{-N}(0)}^0 \Big( w_0(x)+\ovl{u}_l \Big)~dx  - \int_0^{\Gamma_r^N(0)} \Big( w_0(x)+\ovl{u}_r \Big) ~dx \label{I2} \\
	=& (a-Np)\ovl{u}_l-(a+Np)\ovl{u}_r \notag
	\end{align}
	And Corollary \ref{corgamma} implies that 
	\begin{align}
	I_3+I_4=& - \int_0^T \Big\{ f'(\ovl{u}_r)~u_r(\Gamma_r^N(t),t)-f\Big(u_r(\Gamma_r^N(t),t)\Big) \Big\}~dt \label{I34}\\ 
	&+\int_0^T \Big\{ f'(\ovl{u}_l)~u_l(\Gamma_l^{-N}(t),t)-f\Big(u_l(\Gamma_l^{-N}(t),t)\Big) \Big\}~dt \notag\\
	=& - \int_0^T \Big( f'(\ovl{u}_r)\ovl{u}_r-f(\ovl{u}_r) \Big)~dt +\int_0^T \Big( f'(\ovl{u}_l)\ovl{u}_l-f(\ovl{u}_l) \Big)~dt \notag\\
	=& \Big\{ f'(\ovl{u}_l)\ovl{u}_l-f'(\ovl{u}_r)\ovl{u}_r-\Big(f(\ovl{u}_l)-f(\ovl{u}_r)\Big) \Big\}T. \notag
	\end{align}
	Thus by \eqref{I1}, \eqref{I2}, \eqref{I34}, and noting that $\Gamma_l^{-N}(T)= a-Np+f'(\ovl{u}_l)T,~ \Gamma_r^N(T)= a+Np+f'(\ovl{u}_r)T $, one has 
	\begin{align}
	& X(T)- sT \label{Xfor}\\
	& = \frac{-1}{\ovl{u}_l-\ovl{u}_r} \Big\{ \int_{\Gamma_l^{-N}(T)}^{X(T)} \Big(u_l(x,T)-\ovl{u}_l\Big)~dx 
	+\int_{X(T)}^{\Gamma_r^N(T)} \Big(u_r(x,T)-\ovl{u}_r \Big)~dx \Big\}.\notag
	\end{align}
	Then by \eqref{leq} and \eqref{Xfor}, it holds that for $T>T_S, $
	$$ \Big| X(T)-sT| \leq \frac{C}{T}. $$
	Finally, one has
	\begin{align*}
	& \sup_{x<X(t)} |u(x,t)-\ovl{u}_l| +\sup_{x>X(t)} |u(x,t)-\ovl{u}_r| + |X(t)-st| \\
	=& \sup_{x<X(t)} |u_l(x,t)-\ovl{u}_l| +\sup_{x>X(t)} |u_r(x,t)-\ovl{u}_r| + |X(t)-st|\\
	\leq & \frac{C}{t}, \quad \forall~ t>T_S,
	\end{align*} 
	where $C$ depends on $\ovl{u}_l, \ovl{u}_r, p, f.$\\
	This completes the proof of Theorem \ref{lb}.	
\end{proof}

\begin{proof}[Proof of Theorem \ref{thmshock2}]
		When $f(u)=\dfrac{u^2}{2}, $ by Galilean transformation, one has that 
	$$ u_l(x,t)=w(x-\ovl{u}_l t,t)+\ovl{u}_l,~ u_r(x,t)=w(x-\ovl{u}_r t, t)+\ovl{u}_r. $$
	Then by \eqref{Xfor}, it holds that 
	\begin{align*}
	X(t)-st  &= \frac{-1}{\ovl{u}_l-\ovl{u}_r} \Big\{ \int_{\Gamma_l^{-N}(t)}^{X(t)}w(x-\ovl{u}_l t,t)~dx 
	+\int_{X(t)}^{\Gamma_r^N(t)} w(x-\ovl{u}_r t,t)~dx \Big\}\\
	&= \frac{-1}{\ovl{u}_l-\ovl{u}_r} \Big\{ \int_{a-Np}^{X(t)-\ovl{u}_l t} w(y,t)~dy + \int_{X(t)-\ovl{u}_r t}^{a+Np} w(y,t)~dy \Big\}\\
	&= \frac{1}{\ovl{u}_l-\ovl{u}_r}\int_{X(t)-\ovl{u}_l t}^{X(t)-\ovl{u}_r t} w(y,t)~dy
	\end{align*}
	If $ (\ovl{u}_l-\ovl{u}_r)t=np $ for any positive integer $ n, $ then $X(t)=st,$ which means that the perturbed shock $x=X(t) $ will coincide with the background shock $x=st $ after a period of time $\dfrac{p}{\ovl{u}_l-\ovl{u}_r}$. % \sout{ However, for the general convex flux $f(u)$, this refined property for Burger's equation does not hold in general.}\todo{...}
\end{proof}

\vspace{0.8cm}
%%%%%%%%%%%%%%%%%%%%%%%%%%%%%%%%%%%%%%%%%%%
%% rarefaction waves
%%%%%%%%%%%%%%%%%%%%%%%%%%%%%%%%%%%%%%%%%%%
\section{Proof of Theorem \ref{thmrare} and \ref{thmrare2}}
In this section, we will prove Theorem \ref{thmrare} and \ref{thmrare2}.

%%%%%%%%%%%%%%%%%%%%%%%%%%%%%%%%%%%%%%%%%%
%%%%%%%%%%%%%%%%%%%%%%%%%%%%%%%%%%%%%%%%%%
\iffalse
\begin{Lem}\label{lemrare}	
	There exists a finie time $T_R=\dfrac{p}{f'(\ovl{u}_r)-f'(\ovl{u}_l)}>0$, such that 
	\begin{equation}\label{xlxr}
		X_{l+}(t)\leq\Gamma_l^0(t)<\Gamma_r^{-1}(t)\leq X_{r-}(t), \quad t>T_R.
	\end{equation}
\end{Lem}
\begin{proof}[Proof of Lemma \ref{lemrare}
	Since $\Gamma_r^{-1}(t)=f'(\ovl{u}_r)t+a-p$ and $\Gamma_r^0(t)=f'(\ovl{u}_r)t+a$ are two divides corresponding to $u_r(x,t)$ (see Corollary \ref{corgamma}), thus the characteristic $X_{r-}(t)$ cannot run out of the region between these two divides, that is 
    \begin{equation*}
    \Gamma_r^{-1}(t) \leq X_{r-}(t) \leq \Gamma_r^0(t), \quad \forall~ t>0.
    \end{equation*}
    Similar way to verify 
    \begin{equation*}
    \Gamma_l^{-1}(t) \leq X_{l+}(t) \leq \Gamma_l^0(t), \quad \forall~ t>0. 		
    \end{equation*}
    See Figure.\\

    As $\ovl{u}_l<\ovl{u}_r$ and $f(u)$ is convex
	, thus it's easy to prove \eqref{xlxr}). 
\end{proof}
\fi
%%%%%%%%%%%%%%%%%%%%%%%%%%%%%%%%%%%%%%%%%%
%%%%%%%%%%%%%%%%%%%%%%%%%%%%%%%%%%%%%%%%%%

\begin{proof}[Proof of Theorem \ref{thmrare}]
	Since $\Gamma_r^{-1}(t)=f'(\ovl{u}_r)t+a-p$ and $\Gamma_r^0(t)=f'(\ovl{u}_r)t+a$ are two divides corresponding to $u_r(x,t)$ (see Corollary \ref{corgamma}), thus the characteristic $X_{r-}(t)$ associated with $u_r$, which is defined in \eqref{dX}, cannot run out of the region between these two divides, that is 
	\begin{equation*}
	\Gamma_r^{-1}(t) \leq X_{r-}(t) \leq \Gamma_r^0(t), \quad \forall~ t>0.
	\end{equation*}
	And similarly, it holds that 
	\begin{equation*}
	\Gamma_l^{-1}(t) \leq X_{l+}(t) \leq \Gamma_l^0(t), \quad \forall~ t>0. 		
	\end{equation*}
	See Figure \ref{rarefig}.

%%%%%%%%%%%%%%%%%%%%%%%%%%%%%%%%%%%%%%%%%%%%%%%%%%%%%%%%%%%
\begin{figure}[h]
	\includegraphics[scale=0.4]{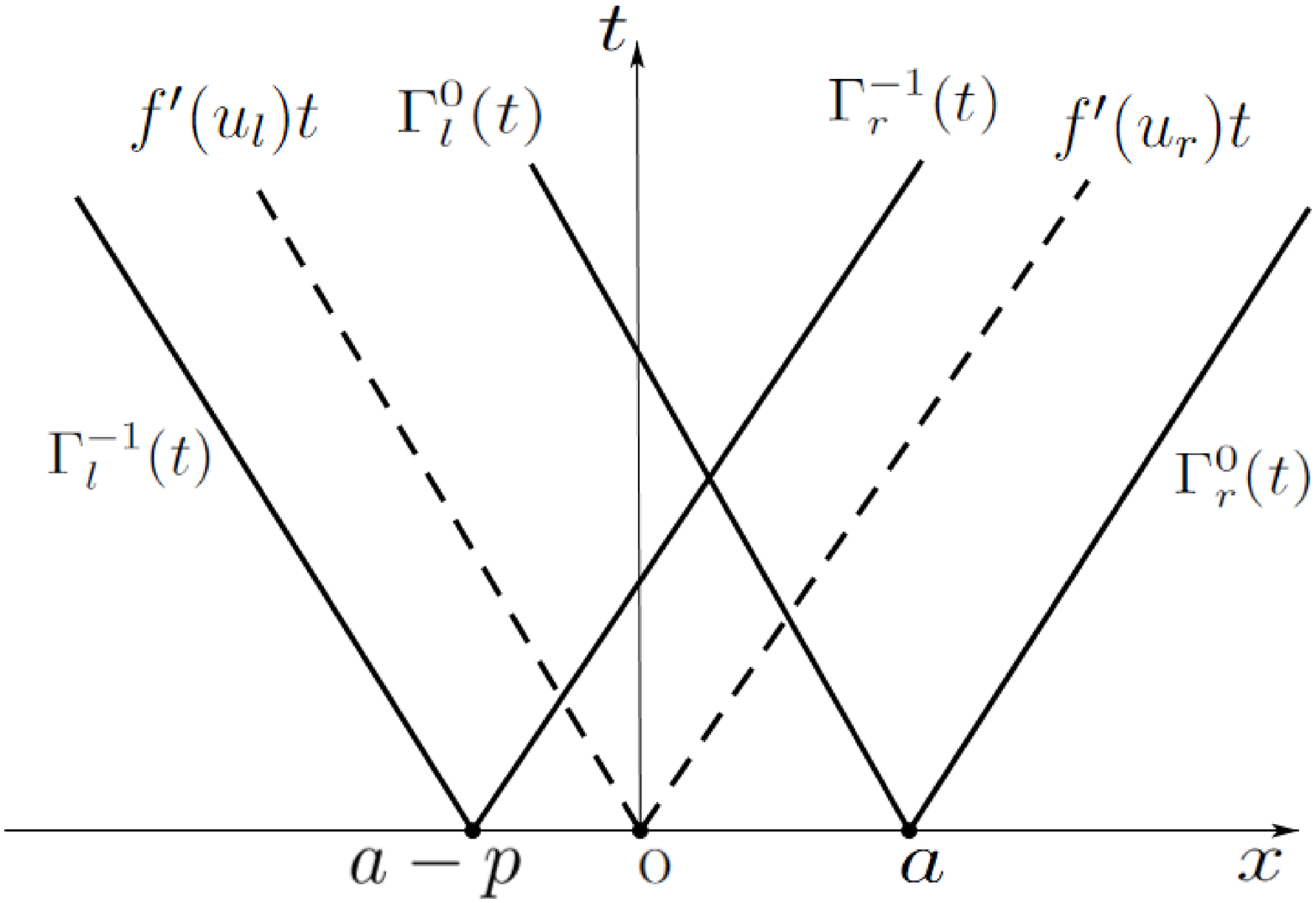} 
	\caption{}
	\label{rarefig}
\end{figure}
%%%%%%%%%%%%%%%%%%%%%%%%%%%%%%%%%%%%%%%%%%%%%%%%%%%%%%%%%%%
	
Now by Proposition \ref{propglue}, if $x<\Gamma_l^{-1}(t)$, then $u(x,t)=u_l(x,t)$; and if $x>\Gamma_r^0(t)$, then $u(x,t)=u_r(x,t)$. And if $\Gamma_l^{-1}(t)< x <\Gamma_r^0(t)$, the following claim holds.

\textbf{Claim:} For $\Gamma_l^{-1}(t)<x<\Gamma_r^0(t)$, the extremal backward characteristics associated with $u$ cannot intersect with $\Gamma_l^{-1}(t)$ nor $\Gamma_r^0(t)$ for $t>0$. 

In fact, as $u(\Gamma_l^{-1}(t)-,t)=u_l(\Gamma_l^{-1}(t)-,t) \equiv \ovl{u}_l$, thus by entropy condition, one has $$ u(\Gamma_l^{-1}(t)+,t) \leq \ovl{u}_l \quad \text{for any } ~t>0, $$ so if there exists any point $ (\ovl{x}, \ovl{t}) $ between $\Gamma_l^{-1}$ and $\Gamma_r^0,$ such that the minimal backward characteristic of $u$ issuing from $ (\ovl{x}, \ovl{t}) $ intersects with $\Gamma_l^{-1}$ at a point $ (\Gamma_l^{-1}(\tau),\tau)$ with $\tau>0$, then its slope $f'(u(\Gamma_l^{-1}(\tau)+,\tau)) >f'(\ovl{u}_l)$ which is the slope of $\Gamma_l^{-1}$ (see Proposition \ref{propgen}), then $u(\Gamma_l^{-1}(\tau)+,\tau)>\ovl{u}_l$ by strict convexity of $f$, which turns out to be a contradiction. Similarly, one can show that maximal backward characteristic issuing from any point between $ \Gamma_l^{-1} $ and $ \Gamma_r^0 $ will not intersect with $\Gamma_r^0$.

Thus by the arguments above, one may conclude that for any fixed $t>0$,
    \begin{enumerate}
		\item[1)] If $x< \Gamma_l^{-1}(t)$, then combined with \eqref{lb2}, one has $|u(x,t)-u^R(x,t)|=|u_l(x,t)-\ovl{u}_l|\leq \dfrac{C}{t}$.
		
		\item[2)] If $\Gamma_l^{-1}(t)< x<f'(\ovl{u}_l)t$, by Claim, one has 
		\begin{align*}
		& \dfrac{x-a}{t}\leq f'(u(x,t))\leq\dfrac{x-(a-p)}{t}\\
        \Longrightarrow \quad & \dfrac{a-p+f'(\ovl{u}_l)t-a}{t}\leq f'(u(x,t))\leq\dfrac{f'(\ovl{u}_l)t-(a-p)}{t}\\
        \Longrightarrow \quad & -\dfrac{p}{t} \leq f'(u(x,t))-f'(\ovl{u}_l)\leq \dfrac{p}{t}\\
		\Longrightarrow \quad & |u(x,t)-\ovl{u}_l| \leq \dfrac{C}{t}
		\end{align*}
		here $a\in[0,p)$ is used. 
		Therefore, $|u(x,t)-u^R(x,t)|=|u(x,t)-\ovl{u}_l|\leq \dfrac{C}{t}$.
		
		\item[3)] If $f'(\ovl{u}_l)t\leq x< f'(\ovl{u}_r)t$, then $ u^R(x,t)=(f')^{-1}(\dfrac{x}{t}) $, and similarly to 2), by Claim, one still has
		\begin{align*}
			& \dfrac{x-a}{t}\leq f'(u(x,t))\leq\dfrac{x-(a-p)}{t}\\
			\Longrightarrow \quad & | f'(u)-f'(u^R) | \leq \dfrac{p}{t}\\
			\Longrightarrow \quad & |u(x,t)-u^R(x,t)| \leq \dfrac{C}{t}
		\end{align*}		
		\item[4)] The other cases are similar.
	\end{enumerate}
Therefore, the proof is finished.
\end{proof}
\vspace{0.5cm}

\begin{proof}[Proof of Theorem \ref{thmrare2}]
	Since for all $x\in \mathbb{R}$, $\int_0^x w_0(y)~dy \geq 0$, thus one can choose $a=0$ in Corollary \ref{corgamma} for this case. Hence, $\Gamma_l^0(t)=f'(\ovl{u}_l)t\leq X_{l+}(t)$ and $X_{r-}(t)\leq \Gamma_r^0(t)=f'(\ovl{u}_r)t$, for all $t\geq 0$. So if $x<f'(\ovl{u}_l)t$, then $u(x,t)=u_l(x,t)$; if $x>f'(\ovl{u}_r)t$, then $u(x,t)=u_r(x,t)$; and if $f'(\ovl{u}_l)t\leq x\leq f'(\ovl{u}_r)t$, by similar arguments below Figure \ref{rarefig}, the extremal backward characteristics emanating from $(x,t)$ cannot cross through $\Gamma_l^0$ or $\Gamma_r^0$ at positive time, and thus both of them have to intersect with the $x$-axis at the origin, and hence $f'(u(x,t))=\dfrac{x}{t}=f'(u^R(x,t))$. The proof is finished.
\end{proof}

\vspace{0.8cm}

%%%%%%%%%%%%%%%%%%%%%%%%%%%%%%%%%%%%%%%%%%%
%% 
%%%%%%%%%%%%%%%%%%%%%%%%%%%%%%%%%%%%%%%%%%%
\section{Alternative proof of Theorem~\ref{thmshock} for Burger's equation}

In this section we present an alternative proof of \eqref{glue} in Theorem~\ref{thmshock} for the Burger's equation with initial data $u_0(x)$ in \eqref{ic1} with $ \ovl{u}_l > \ovl{u}_r, $ and $ w_0 $ satisfies \eqref{ave0}. This method depends on the Hopf's solution given in \cite{hopf}.

Denote the viscous solution $u^\e(x,t)$ to the viscous equation 
\begin{equation}\label{equ2}
	\p_t u^\e+\p_x \Big( \frac{(u^\e)^2}{2} \Big)= \e \p_x^2 u^\e, \quad x \in (-\infty, +\infty),\quad t>0,
\end{equation} with initial data \eqref{ic1} $ u^\e(x,0)=u_0(x)$. \\
By Hopf-Cole transformation $u^\e$ can be computed in an explicit formula,
\begin{equation}\label{hop}
	u^\e(x,t)=\int_{-\infty}^{\infty} \frac{x-y}{t} e^{-F(t,x,y)/2\e} dy \Big{/} \int_{-\infty}^{\infty} e^{-F(t,x,y)/2\e} dy.
\end{equation}
where $F(t,x,y)\de\frac{(x-y)^2}{2t}+\int_0^y u_0(z)dz.$
Let $\e \rightarrow 0, $ it is well known that $ u^\e(x,t) $ converges to the unique entropy solution $u(x,t)$ to \eqref{equ1}, \eqref{ic1} almost everywhere.\\

Before giving the proof, we need some notations.\\
Denote
\begin{align}
	& F_l(t,x,y) \de \frac{(x-y)^2}{2t}+\int_0^y \ovl{u}_l+w_0(z) \ dz;\label{fl}\\
	& F_r(t,x,y) \de \frac{(x-y)^2}{2t}+\int_0^y \ovl{u}_r+w_0(z) \ dz;\label{fr}\\
	& F(t,x,y) \de \begin{cases}
		F_l(t,x,y), \quad \text{~if } y \leq 0;\\
		F_r(t,x,y), \quad \text{~if~ } y \geq 0.
	\end{cases}\label{f}\\
	& Y_{l*}(t,x) \de \min_{z \in \R} \left\{ z: F_l(t,x,z)=\ml \right\};\label{ylb}\\
	& Y_l^*(t,x) \de \max_{z \in \R} \left\{ z: F_l(t,x,z)=\ml \right\};\label{ylt}\\
	& Y_{r*}(t,x) \de \min_{z \in \R} \left\{ z: F_r(t,x,z)=\mr \right\};\label{yrb}\\
	& Y_r^*(t,x) \de \max_{z \in \R} \left\{ z: F_r(t,x,z)=\mr \right\};\label{yrt}\\	
	& Y_*(t,x) \de \min_{z \in \R} \left\{ z: F(t,x,z)=\m \right\};\label{yb}\\
	& Y^*(t,x) \de \max_{z \in \R} \left\{ z: F(t,x,z)=\m \right\};\label{yt}\\
	& m_-(t,x) \de \mln  \quad \text{and} \quad m_+(t,x) \de \mrp. \label{defm}
\end{align}
As in \cite{hopf}, one has the following lemma
\begin{Lem}\label{lem1} The following properties hold:
	\begin{enumerate}
		\item[(i).] $ m_-(t,x) $ and $ m_+(t,x) $ are both continuous in $  t>0, x \in \R. $
		
		\item[(ii).] $ Y_{l*}(t,x),~ Y_{r*}(t,x),~ Y_*(t,x)$ are increasing and continuous to the left with respect to $ x, $ for any $ t>0. $
		
		\item[(iii).] $ Y_l^*(t,x),~ Y_r^*(t,x),~ Y^*(t,x) $ are increasing and continuous to the right with respect to $ x, $ for any $ t>0. $
		
		\item[(iv).] If $ x_1<x_2, $ then 
		\begin{align*}
			Y_l^*(t,x_1) \leq Y_{l*}(t,x_2),\quad
			Y_r^*(t,x_1) \leq Y_{r*}(t,x_2),\quad
			Y^*(t,x_1) \leq Y_*(t,x_2).
		\end{align*}
	\end{enumerate}

\end{Lem}
\begin{proof}
	(i) can be proved easily by the fact that $F_r(t,x,y), F_l(t,x,y) $ are both continuous in $t,x,y.$ And (ii), (iii), (iv) are derived from Lemma 1 in \cite{hopf}.
\end{proof}

\begin{Prop}[Theorem 3 in \cite{hopf}]\label{prophopf}
	under the assumptions of Theorem \ref{thmshock}, it holds that for almost all $x \in \R, t>0, $ 
	\begin{align}
		& u_l(x+,t) = \frac{x- Y_l^*(x,t)}{t}, ~ u_l(x-,t) = \frac{x- Y_{l*}(x,t)}{t},\\
		& u_r(x+,t) = \frac{x- Y_r^*(x,t)}{t}, ~ u_r(x-,t) = \frac{x- Y_{r*}(x,t)}{t},\\ 
		& u(x+,t) = \frac{x- Y^*(x,t)}{t}, ~ u(x-,t) = \frac{x- Y_*(x,t)}{t}.
	\end{align}
\end{Prop}

\begin{proof}[Proof of Theorem \ref{thmshock}]
	Since $ w_0(x) $ is $ L^{\infty} $ bounded, there exist two constant numbers $ \alpha < \beta, $ such that
	$$\alpha < w_0(x) <\beta, \quad \quad a.e.~ x.$$
	
	Now we compare $$ m_-(t,x)=\mln  \text{~~ with ~~} m_+(t,x)= \mrp. $$

	\begin{enumerate}
		\item[\textbf{Case1.}]
		If $~\dfrac{x}{t} < s+\alpha, $ where $s=\dfrac{\ovl{u}_l+\ovl{u}_r}{2}. $
		Then 
		\begin{align}
			m_-(t,x)& \leq \min_{y\leq 0} \Big(  \frac{(y-x)^2}{2t}+(\ovl{u}_l+\alpha)y \Big) \label{fyl}  \\
			& = \min_{y\leq 0} \Big(\frac{1}{2t} \{y-[x-(\ovl{u}_l+\alpha)t]\}^2+(\ovl{u}_l+\alpha)x-\frac{(\ovl{u}_l+\alpha)^2}{2}t \Big)  \notag \\
			& = (\ovl{u}_l+\alpha)x-\frac{(\ovl{u}_l+\alpha)^2}{2}t ; \notag\\
			m_+(t,x) & \geq \min_{y\geq 0} \Big(  \frac{(y-x)^2}{2t}+(\ovl{u}_r+\alpha)y \Big) \label{fyr}\\
			& = \min_{y\leq 0} \Big(\frac{1}{2t} \{y-[x-(\ovl{u}_r+\alpha)t]\}^2+(\ovl{u}_r+\alpha)x-\frac{(\ovl{u}_r+\alpha)^2}{2}t \Big) \notag \\
			& \geq (\ovl{u}_r+\alpha)x-\frac{(\ovl{u}_r+\alpha)^2}{2}t. \notag
		\end{align}
		Using \eqref{fyl} and \eqref{fyr}, one has 
		\begin{align*}
			m_+(t,x) - m_-(t,x) & \geq  (\ovl{u}_r+\alpha)x-\frac{(\ovl{u}_r+\alpha)^2}{2}t - (\ovl{u}_l+\alpha)x + \frac{(\ovl{u}_l+\alpha)^2}{2}t \\
			&	= (\ovl{u}_l - \ovl{u}_r) \big((s+\alpha)t - x\big) >0,
		\end{align*}
		so in this case, 
		\begin{equation}\label{neg}
			m_+(t,x) > m_-(t,x).
		\end{equation}
		
		\item[\textbf{Case2.}]
		If $~\dfrac{x}{t} >s+\beta,$ by similar argument as in Case 1, one can prove that 
		\begin{equation}\label{pos}
			m_+(t,x) < m_-(t,x).
		\end{equation}\\ 
	\end{enumerate}

	It then follows from \eqref{neg}, \eqref{pos}, and the continuity of $	m_{\pm}(t,x) $ that there must exist a closed nonempty set for each $t>0, $
	\begin{equation}\label{set}
		\X(t)\de\{ x: m_-(t,x)=m_+(t,x) \} \subset [~(s+\alpha) t, (s+\beta) t~]
	\end{equation}
	Define the minimum value and the maximum value in $\X(t) $ as:
	\begin{equation}\label{Xmm}
		X_-(t) \de \min\{x: x\in \X(t)\}, \quad  X_+(t) \de \max\{x: x\in \X(t)\}
	\end{equation}
	Since $\X(t) $ is closed, then $X_{\pm}(t) \in \X(t).$\\
	
	Next, we prove some properties about $ \X(t), X_-(t)$ and $ X_+(t). $
	\begin{Lem}\label{lem2}
		For any $ x \in \X(t), $ it holds that 
		\begin{align}
			& Y_*(t,x)=Y_{l*}(t,x) \leq 0,~ u(x-,t)=u_l(x-,t), \label{Xbn}\\
			& Y^*(t,x)=Y_r^{*}(t,x) \geq 0,~ u(x+,t)=u_l(x+,t),\label{Xtp}			
		\end{align}
	\end{Lem}
	\begin{proof}[Proof of Lemma \ref{lem2}]
	It follows from the definition of $\X(t), $ that $$m_-(t,x)=m_+(t,x), $$	 
	\begin{equation}\label{lre}
		 i.e.\quad \quad \min_{y\leq 0} F_l(t,x,y)=\min_{y\geq 0} F_r(t,x,y).
		\end{equation}
		This together with the definitions of $F(t,x,y) $ in \eqref{f} and $Y_*(t,x) $ in \eqref{yb}, implies that $$ Y_*(t,x) \leq 0.$$
		Due to \eqref{fl}, \eqref{fr}, and $\ovl{u}_r < \ovl{u}_l, $ it's easy to verify that 
		$$ \min_{y\geq 0} F_l(t,x,y) \geq \min_{y\geq 0} F_r(t,x,y) $$
		Then	it holds that
		\begin{equation}\label{inel}
			\min_{y\geq 0} F_l(t,x,y) \geq \min_{y\leq 0} F_l(t,x,y)
		\end{equation}
		which implies that $$Y_{l*}(t,x) \leq 0. $$ 
		From above, we have proved that the minimum values $$\min_{y \in \R} F(t,x,y) \quad \text{and} \quad \min_{y \in \R} F_l(t,x,y)$$ can be achieved in $\{y\leq 0\},$ where $ F(t,x,y) = F_l(t,x,y). $ So it follows that $ Y_*(t,x)=Y_{l*}(t,x), $ and by Proposition \ref{prophopf}, \eqref{Xbn} can be verified easily.\\
		The proof of \eqref{Xtp} is similar.
	\end{proof}

	\begin{Lem}\label{lem3} The following properties hold:
		\begin{enumerate}
			\item[(i).] If $  x<X_-(t), $ then
			\begin{align*}
				& Y_*(t,x)=Y_{l*}(t,x), Y^*(t,x)=Y_l^*(t,x)<0,\\
				 \text{ and hence, } &  u(x,t)=u_l(x,t).
			\end{align*}
			\item[(ii).] If $ x>X_+(t), $ then
			\begin{align*}
				 & Y^*(t,x)=Y_r^*(t,x), Y_*(t,x)=Y_{r*}(t,x)>0,\\
				\text{ and hence, } &  u(x,t)=u_r(x,t).
			\end{align*}
			\item[(iii).] If $ X_-(t) <X_+(t),  $ then $ \forall~ x \in (X_-(t), X_+(t)), $
			\begin{align*}
				& Y_{l*}(t,x)=Y_l^*(t,x) =Y_{r*}(t,x)=Y_r^*(t,x)=Y_*(t,x)=Y^*(t,x)=0,\\
				 \text{ and hence, }& u(x,t)=u_l (x,t)=u_r(x,t)=\frac{x}{t}.
			\end{align*}			
		\end{enumerate}
	\end{Lem}
	\begin{proof}[Proof of Lemma \ref{lem3}]  		$ $\\
		\begin{enumerate}
			\item[(i).] It follows from Lemma \ref{lem2} that 
			$ Y_*(t,X_-(t))=Y_{l*}(t,X_-(t)) \leq 0. $
			Thus by  Lemma \ref{lem1}.(iv) that, if $x<X_-(t), $  then $ Y^*(t,x)\leq Y_*(t,X_-(t)) \leq 0. $ \\
			If $ Y^*(t,x)=0, $ by the definition of $ Y^* $ in \eqref{yt}, it holds that 
			\begin{align*}
				& \min_{y \in \R } F(t,x,y)= \min_{y \leq 0} F_l(t,x,y) = F_l(t,x,0), \\
				 \text{and~ }& \min_{y \in \R} F(t,x,y)<F_r(t,x,z), ~\forall~ z>0. 		
			\end{align*}		 
			Thus $$ \min_{y \leq 0} F_l(t,x,y) = F_l(t,x,0) = F_r(t,x,0) = \min_{y \geq 0} F_r(t,x,y),$$
			which implies that $x \in \X(t) $ defined by \eqref{set}. But $x<X_-(t) $ contradicts with the definition of $X_-(t)$ in \eqref{Xmm}, so it holds that
			\begin{equation}\label{ineq1}
				Y^*(t,x)<0.
			\end{equation}
			By \eqref{ineq1}, there must hold 
			$$ \min_{y \leq 0} F_l(t,x,y) < \min_{y \geq 0} F_r(t,x,y), $$
			then by $$ \min_{y \geq 0} F_r(t,x,y) \leq \min_{y \geq 0} F_l(t,x,y), $$
			one has 
			$$ \min_{y \leq 0} F_l(t,x,y) < \min_{y \geq 0} F_l(t,x,y), $$
			so it holds that
			\begin{equation}\label{ineq2}
				Y_l^*(t,x)<0.
			\end{equation}\\
			By \eqref{ineq1} and \eqref{ineq2}, the minimum values of $ F(t,x,y) $ and $ F_l(t,x,y) $ can only be achieved in $ \{y<0\}, $ where $ F(t,x,y) = F_l(t,x,y), $ thus it holds that $$ Y_l^*(t,x)=Y^*(t,x),~ Y_{l*}(t,x)=Y_*(t,x). $$\\
			
			\item[(ii).] The proof is similar to (i).\\
			
			\item[(iii).] By Lemma \ref{lem2}, one has
			\begin{align*}
			& Y_*(t,X_+(t))=Y_{l*}(t,X_+(t)) \leq 0,\\
			& Y^*(t,X_-(t))=Y_r^{*}(t,X_-(t)) \geq 0.			
			\end{align*}
			Thus if $X_-(t)<x<X_+(t), $ then by Lemma \ref{lem1}.(iv),
			\begin{equation*}
			0\leq Y^*(t,X_-(t)) \leq Y_*(t,x) \leq Y^*(t,x) \leq Y_*(t,X_+(t)) \leq 0
			\end{equation*}
			which implies that 
			\begin{equation}\label{eq2}
			Y_*(t,x)=Y^*(t,x)=0, \quad \forall~ x \in (X_-(t), X_+(t)).
			\end{equation}
			It implies that
			$$\min_{y\in \R} F(t,x,y)=F(t,x,0), $$
			 and 
			\begin{equation*}
			\begin{cases}
			F_l(t,x,y)>F(t,x,0), \quad \forall~ y<0,\\
			F_r(t,x,y)>F(t,x,0), \quad \forall~ y>0,
			\end{cases}
			\end{equation*}
			which implies that
			\begin{equation*}
			\min_{y\leq 0} F_l(t,x,y) =F_l(t,x,0)=F_r(t,x,0)=\min_{y \geq 0} F_r(t,x,y) =\frac{x^2}{2t}.
			\end{equation*}
			Thus, $x \in \X(t) $. \\
			Then by Lemma \ref{lem2} again and by \eqref{eq2}, one has that $Y_{l*}(t,x) =Y_*(t,x)=0.$\\
			Then one has
			$$ 0=Y_{l*}(t,x) \leq Y_l^*(t,x) \leq Y_{l*}(t,X_+(t)) \leq 0, $$
			so $ Y_l^*(t,x)=0.$\\
			Similarly, one can also prove that $$ Y_r^*(t,x)=Y_{r*}(t,x)=0, \text{~for ~} X_-(t)<x<X_+(t). $$
			Due to Proposition \ref{prophopf}, the rest of the Lemma follows easily, which completes the proof.
		\end{enumerate}
	\end{proof}

	To prove that after a finite time $ T_S $, $ X_-(t) = X_+(t). $ it's easy by using Lemma \ref{lem3} (iii) and \eqref{lb2}.
	
%	In fact, take $T_S=\dfrac{2C}{\ovl{u}_l-\ovl{u}_r} $, here $ C $ is the constant in \eqref{lb2}. Then by \eqref{lb2}, if $  t>T_S, $
%	\begin{equation}\label{neq}
%		u_l(x,t) \geq \ovl{u}_l-\frac{C}{t} > \ovl{u}_r+\frac{C}{t} \geq u_r(x,t), \quad a.e.~ x \in \R.
%	\end{equation}
%	Thus the case (iii) in Lemma \ref{lem3} never appears if $t>T_S$, which implies that $X_-(t)=X_+(t), \text{~denoted by } X(t). $ Thus we have shown that 
%	\begin{equation}\label{uni}
%		 \X(t)=\{X(t)\}, \quad \forall~ t>T_0. 
%	\end{equation}

	Next, we prove that when $t>T_S,$ the unique point $X(t) $ in $ \X(t) $ is Lipschitz with respect to $ t. $\\
	
	In fact, it is well-known that $u(x,t) \in Lip~((0,+\infty),L^1_{loc}). $ Thus there exists a positive constant $ C $ such that for any $t>\tau>T_S, $ it holds that
	\begin{equation}\label{intlip}
		\int_{X(t)}^{X(\tau)} |u(x,t)-u(x,\tau)| ~dx \leq C|t-\tau|,
	\end{equation}
	here one has assumed that  $X(\tau)>X(t) $ without loss of generality.
	
	When $X(t)<x<X(\tau), $ by Lemma \ref{lem3}, 
	$$ u(x,t)=u_r(x,t),\quad u(x,\tau) = u_l(x,\tau). $$
	Then \eqref{intlip} yields
	$$ C|t-\tau| \geq \int_{X(t)}^{X(\tau)} (u_l(x,\tau)-u_r(x,t))~dx \geq \Big(\ovl{u}_l-\ovl{u}_r-\dfrac{2C}{T_S}\Big)~|X(t)-X(\tau)|, $$ i.e.
	\begin{equation}\label{Xlip}
		|X(t)-X(\tau)|\leq C(p,\ovl{u}_l,\ovl{u}_r)|t-\tau|, \quad \forall~ t>\tau>T_S.
	\end{equation}
	Since for $t>T_S$ large enough, $X(t) $ is Lipschitz and it is a discontinuous curve of the entropy solution $ u $, so $X(t) $ is actually a shock when $t>T_S. $ \\	
	Then by Lemma \ref{lem2} and Lemma \ref{lem3}.(i), (ii), one can finish the proof of \eqref{glue} .  	
\end{proof} 

\vspace{0.8cm}

\appendix
\section{}
For any $L^{\infty}$ periodic initial data $u_0$ with the average $\ovl{u}$ defined as \eqref{avg}, before showing the theorem of the optimal decay of the corresponding entropy solution $u$, we define two functions $g$ and $z(t)$ associated with $f$ and $\ovl{u}$ as follows:

By changing variables if necessary, we can assume without loss of generality that $$f(\ovl{u})= f'(\ovl{u})=0.$$ Since $f$ is strictly convex, $f'$ is monotonically increasing, so one can define 
\begin{equation}\label{dg}
	g(v)\de \int_0^v [(f')^{-1}(s)-\ovl{u}]~ds,
\end{equation}
here $(f')^{-1}$ represents the inverse function of $f'$.
Therefore, it follows that $g\in C^2 $ and satisfies
\begin{align*}
	& g'(v)=(f')^{-1}(v)-\ovl{u}, \quad g(0)=g'(0)=0,\\
	& g''(v)=1/f''\Big((f')^{-1}(v)\Big)>0
\end{align*} 
which implies that
\begin{equation}\label{ineq3}
	g(0)-g(-\frac{p}{t})<0, \quad  g(\frac{p}{t})-g(0)>0.
\end{equation}

While for any fixed $t>0$, $g(\dfrac{z}{t})-g(\dfrac{z-p}{t})$ is strictly increasing with respect to $z$, then by \eqref{ineq3}, there exists a unique point $z(t) \in (0,p)$ such that
\begin{equation}\label{dz}
g(\frac{z(t)}{t})=g(\frac{z(t)-p}{t})
\end{equation}
And by implicit function theorem, $z(t) \in C^2((0,+\infty))$. 

\begin{Thm} \label{thmper}
	For any periodic initial data $u_0(x)\in L^{\infty}(\R)$ with period $p$, the entropy solution $u(x,t)$ to \eqref{equ1} is also space-periodic of period $p$ for any $t>0$, and it satisfies
	\begin{equation}\label{lb1}
	(f')^{-1}\Big(\frac{z(t)-p}{t}\Big)\leq u(x,t) \leq (f')^{-1}\Big(\frac{z(t)}{t}\Big),~~\forall~ t>0, ~a.e.~x,
	\end{equation} 
	where $\ovl{u}$ is defined as \eqref{avg}, $(f')^{-1}$ is the inverse function of $f'$ and $z(t) \in (0,p)$ is defined as \eqref{dz}.
	More precisely, by the definition of $z(t)$, \eqref{lb1} can imply that
	\begin{equation}\label{lb11}
	\begin{aligned}
	& z(t)=\frac{p}{2}+o(1), \quad \text{as~} t \rightarrow +\infty,\\
	& |u(x,t)-\ovl{u}|\leq \frac{p}{2f''(\ovl{u})t}+o(\frac{1}{t}), \quad \text{as~} t \rightarrow +\infty.
	\end{aligned}	
	\end{equation}
	Furthermore, there exist periodic initial data such that for any $t$ larger than a constant $T_P>0$,
	\begin{equation}\label{lb12}
	\begin{aligned}
	&\inf_{x\in \mathbb{R}} u(x,t) = (f')^{-1}(\frac{z(t)-p}{t}),\\
	&\sup_{x\in \mathbb{R}} u(x,t) = (f')^{-1}(\frac{z(t)}{t}).
	\end{aligned}
	\end{equation}
\end{Thm}

\vspace{3cm}

By a translation of $x$-axis, one can also assume in \eqref{defa} that the integral minimal point $a=0$, since the equation \eqref{equ1} is invariant under this translation. Thus by Proposition \ref{propdiper} and $f'(\ovl{u})=0 $, it holds that for $N \in \mathbb{Z}$, the straight lines $x=Np$ are all divides of the periodic entropy solution $u(x,t)$.

Then since $u(x,t)$ is space-periodic of period $p$ at any time, we can just focus on the region between two divides $x=0$ and $x=p$.

\begin{proof}[Proof of Theorem \ref{thmper}]
	\textbf{Step1. } We first prove that the bounds for $u(x,t)$ in \eqref{lb1} can be attained, i.e. there exists an initial data $u_0(x)$ with the average $\ovl{u}$, s.t. \eqref{lb12} holds.
	
	Let $m_1, m_2>0$ be any constants and define $u_0(x)$ in one period $(0,p)$ as 	\begin{equation}\label{2constants}
	u_0(x)=\begin{cases}
	m_1+\ovl{u} & \text{if}~0<x<\frac{m_2}{m_1+m_2}p,\\
	-m_2+\ovl{u} & \text{if}~ \frac{m_2}{m_1+m_2}p<x<p.
	\end{cases}
	\end{equation}
	This function is piecewise constant and the average is $\ovl{u}$. Thus by the assumption $f'(\ovl{u})=0, f''>0$, we have $f'(-m_2+\ovl{u})<0<f'(m_1+\ovl{u})$. Thus it is easy to verify that the forward generalized characteristic issuing from the point $(\frac{m_2}{m_1+m_2}p,0)$ of $x$-axis is unique, denoted by $x=\zeta(t)$, which is a shock for short time.
	
	\textbf{Claim}: After a finite time 
	$$T_P\de \max\{\frac{p}{f'(m_1+\ovl{u})},\frac{p}{-f'(-m_2+\ovl{u})}\},$$
	the minimal (resp.~ maximal) backward characteristic emanating from $(\zeta(t),t)$ intersects with $x$-axis at the origin (resp.~ $(p,0)$).
	See Figure \ref{2rarefig}.
	
	%%%%%%%%%%%%%%%%%%%%%%%%%%%%%%%%%%%%%%%%%%%%%%
	\begin{figure}[h]
		\includegraphics[scale=0.35]{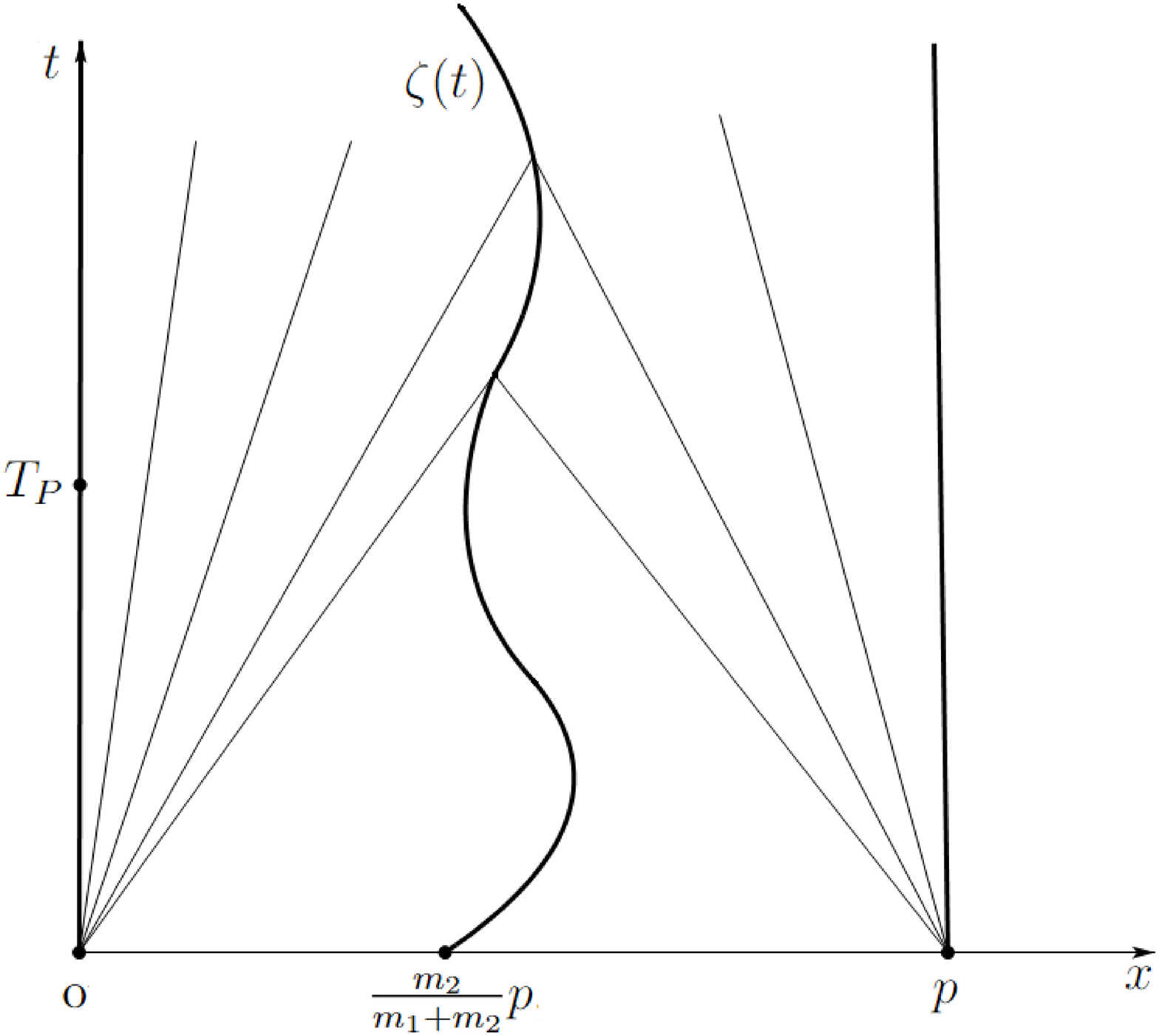} 
		\caption{}
		\label{2rarefig}
	\end{figure}
	%%%%%%%%%%%%%%%%%%%%%%%%%%%%%%%%%%%%%%%%%%%%%%
	
	Indeed, for fixed $\ovl{t}>T_P$, emanating from $(\zeta(\ovl{t}), \ovl{t}),$ the minimal backward characteristic $\xi_-(t)$ of $u$
	$$\xi_-(t)\de f'\Big(u(\zeta(\ovl{t})-, \ovl{t})\Big)(t-\ovl{t})+\zeta(\ovl{t}) $$
	cannot intersect with the divide $x=0$ at $t>0$, and also $ \xi_-(t) \leq \zeta(t) $ since $ \xi_-(t) $ is the minimal generalized characteristic, thus we have $0 \leq \xi_-(0)\leq \zeta(0)=\dfrac{m_2 p}{m_1+m_2} $ and $f'(u(\zeta(\ovl{t})-, \ovl{t}))\leq p/\ovl{t}. $ Then by \eqref{genpro}, it holds that
	\begin{equation}\label{exine}
	u_0(\xi_-(0)-)\leq u(\zeta(\ovl{t})-,\ovl{t}) \leq (f')^{-1}(\frac{p}{\ovl{t}})<(f')^{-1}(\frac{p}{T_P})\leq m_1+\ovl{u}.
	\end{equation}
	If $0<\xi_-(0)\leq \dfrac{m_2 p}{m_1+m_2}$, then  by the definition of $u_0$ in \eqref{2constants}, $u_0(\xi_-(0)-)=m_1+\ovl{u}$, which contradicts with \eqref{exine}. So $\xi_-(0)=0 $ must hold.
	The arguments for the maximal backward characteristic is similar, which proves the Claim.
	
	For $t>T_p$, similar to the proof of \eqref{exine}, one can verify also that
	\begin{equation}\label{2rare}
	\begin{aligned}
	& \text{~if~ } 0<x<\zeta(t),\quad u(x-,t)<m_1+\ovl{u}\\
	& \text{~if~ } \zeta(t)<x<p,\quad u(x+,t)>-m_2+\ovl{u}.	
	\end{aligned}
	\end{equation}
	Thus \eqref{2rare} implies that for $t>T_p$, if $0<x<\zeta(t)$, the backward generalized characteristic emanating from $(x,t)$ can only intersect with $x$-axis at origin, which means that the backward generalized characteristic (a real characteristic) is unique, so $u(x,t)=(f')^{-1}(\frac{x}{t})$; and respectively, if $\zeta(t)<x<p$, the characteristic can only intersect with $x$-axis at $(p,0)$, so $u(x,t)=(f')^{-1}(\frac{x-p}{t}) $. See Figure \ref{2rarefig}.
	
	Then it follows
	\begin{equation*}
	\begin{aligned}
	0&=\int_0^p [u(y,t)-\ovl{u}] \ dy\\
	&=\int_0^{\zeta(t)} [(f')^{-1}(\frac{y}{t})-\ovl{u}] \ dy+\int_{\zeta(t)}^p [(f')^{-1}(\frac{y-p}{t})-\ovl{u}] \ dy\\
	&=\Big[g(\frac{\zeta(t)}{t})-g(\frac{\zeta(t)-p}{t})\Big]t.
	\end{aligned}
	\end{equation*}
	Therefore, after $t>T_P$, $\zeta(t)=z(t) \in (0,p)$, and  $ u(z(t)-,t)=(f')^{-1}(\frac{z(t)}{t})$ and $ u(z(t)+,t)=(f')^{-1}(\frac{z(t)-p}{t})$, which achieves \eqref{lb12}.
	
	\vspace{0.5cm}
	
	\textbf{Step2. } We prove \eqref{lb1} by a contradiction argument. 
	\begin{enumerate}
		\item[1)] Suppose that there exist $x\in(0,p)$ and $t>0$ such that $u(x-,t)>(f')^{-1}(\frac{z(t)}{t})$ and the minimal backward characteristic emanating from $(x,t)$ is $\xi(\tau), \tau\in[0,t]$.
		
		Denote $\lambda\de \xi(0)$ and $\mu\de x-\lambda-z(t)$, and thus $u(x-,t)=(f')^{-1}(\frac{x-\lambda}{t})$. $u(x-,t)>(f')^{-1}(\frac{z(t)}{t})$ implies $\mu>0$, while $x\in(0,p)$ implies $0\leq \lambda < p-z(t)-\mu$.
		
		Note that when $0<y<x$, the maximal backward characteristics emanating from $(y,t)$ cannot cross $\xi(\tau)$, thus $u(y+,t)\geq (f')^{-1}(\frac{y-\lambda}{t})$; 
		when $x<y<p$, the maximal backward characteristics emanating from $(y,t)$ cannot cross $x=p$, thus $u(y+,t) \geq (f')^{-1}(\frac{y-p}{t})$.
		Therefore, one has
		\begin{equation}\label{0geq}
		\begin{aligned}
		0&=\int_0^x [u(y,t)-\ovl{u}] \ dy+\int_x^p [u(y,t)-\ovl{u}] \ dy\\
		&\geq\int_0^x [(f')^{-1}(\frac{y-\lambda}{t})-\ovl{u}] \ dy+\int_x^p [(f')^{-1}(\frac{y-p}{t})-\ovl{u}] \ dy\\
		&=\Big[g(\frac{x-\lambda}{t})-g(-\frac{\lambda}{t})-g(\frac{x-p}{t})\Big]t\\
		&=\Big[g(\frac{z(t)+\mu}{t})-g(-\frac{\lambda}{t})-g(\frac{z(t)+\mu-p}{t}+\frac{\lambda}{t})\Big]t.
		\end{aligned}
		\end{equation}
		As $g$ is convex and $p-z(t)-\mu>\lambda \geq 0$, one has
		\begin{equation}\label{2ineq}
		\begin{aligned}
		g(-\frac{\lambda}{t})
		&\leq \frac{\lambda}{p-z(t)-\mu}~g(\frac{z(t)+\mu-p}{t})+\Big(1-\frac{\lambda}{p-z(t)+\mu}\Big)~g(0)\\
		&=\frac{\lambda}{p-z(t)-\mu}~g(\frac{z(t)+\mu-p}{t})\\
		g(\frac{z(t)+\mu-p}{t}+\frac{\lambda}{t})
		&\leq \Big(1-\frac{\lambda}{p-z(t)-\mu}\Big)~g(\frac{z(t)+\mu-p}{t})+\frac{\lambda}{p-z(t)+\mu}~g(0)\\
		&=\Big(1-\frac{\lambda}{p-z(t)-\mu}\Big)~g(\frac{z(t)+\mu-p}{t})
		\end{aligned}
		\end{equation}
		Therefore, taking \eqref{2ineq} into \eqref{0geq} leads to
		\begin{equation*}
		g(\frac{z(t)+\mu}{t})-g(\frac{z(t)+\mu-p}{t}) \leq 0.
		\end{equation*}		
		But by the definition of $z(t)$ in \eqref{dz} and $g(\frac{z}{t})-g(\frac{z-p}{t})$ is strictly increasing with respect to $z$, this is a contradiction with $\mu>0$. \\
		Hence for any $ x \in (0,p),~ t>0, $ it holds that
		$$ u(x-,t) \leq (f')^{-1}(\frac{z(t)}{t}). $$
		
		\item[2)] Suppose that there exist $x\in(0,p)$ and $t>0$ such that $u(x+,t)<(f')^{-1}(\frac{z(t)-p}{t})$ and the maximal backward characteristic emanating from $(x,t)$ is $\xi(\tau), \tau\in[0,t]$.
		
		Denote $\lambda\de \xi(0)$ and $-\mu\de x-\lambda-z(t)+p$, and thus $u(x+,t)=(f')^{-1}(\frac{x-\lambda}{t})$. $u(x+,t)<(f')^{-1}(\frac{z(t)-p}{t})$ implies $\mu>0$, while $x\in(0,p)$ implies $0\leq p- \lambda < z(t)-\mu$. 
		
		Note that when $0<y<x$, the maximal backward characteristics emanating from $(y,t)$ cannot cross $x=0$, thus $u(y+,t)\leq (f')^{-1}(\frac{y}{t})$; 
		when $x<y<p$, the maximal backward characteristics emanating from $(y,t)$ cannot cross $\xi(\tau)$, thus $u(y+,t)\leq (f')^{-1}(\frac{y-\lambda}{t})$.
		
		Therefore, in the similar way as in 1), one can obtain
		\begin{equation*}
		\begin{aligned}
		0&=\int_0^x [u(y,t)-\ovl{u}] \ dy+\int_x^p [u(y,t)-\ovl{u}] \ dy\\
		&\leq\int_0^x [(f')^{-1}(\frac{y}{t})-\ovl{u}] \ dy+\int_x^p [(f')^{-1}(\frac{y-\lambda}{t})-\ovl{u}] \ dy\\
		&=\Big[g(\frac{x}{t})+g(\frac{p-\lambda}{t})-g(\frac{x-\lambda}{t})\Big]t\\
		&=\Big[g(\frac{z(t)-\mu}{t}-\frac{p-\lambda}{t})+g(\frac{p-\lambda}{t})-g(\frac{x-\lambda}{t})\Big]t\\
		&\leq \Big[g(\frac{z(t)-\mu}{t})+g(0)-g(\frac{x-\lambda}{t})\Big]t \quad \quad \quad (\text{similar to proof of \eqref{2ineq}})    \\
		&= \Big[g(\frac{z(t)-\mu}{t})-g(\frac{z(t)-p-\mu}{t})\Big]t <0, \quad \quad \quad (\mu>0)
		\end{aligned}
		\end{equation*}
		which is also a contradiction.\\
		Hence for any $ x \in (0,p),~ t>0, $ it holds that
		$$ u(x+,t) \geq (f')^{-1}(\frac{z(t)-p}{t}). $$
	\end{enumerate}
	Combining 1), 2) and using the entropy condition $ u(x-,t) \geq u(x+,t), $ one can prove \eqref{lb1}.
	\vspace{0.5cm}
	
	\textbf{Step3. } By \eqref{dz} and Taylor expansion, one has
	\begin{align*}
	& \frac{1}{2}g''(0)\Big[\frac{z(t)}{t}\Big]^2 = \frac{1}{2}g''(0)\Big[\frac{z(t)-p}{t}\Big]^2+o(\frac{1}{t^2}) \qquad \text{as}\ t\rightarrow +\infty,\\
	\Rightarrow~~ & z(t)=\frac{p}{2}+o(1) \qquad \text{as}\ t\rightarrow +\infty.
	\end{align*}
	thus
	\begin{equation*}
	\begin{aligned}
	(f')^{-1}(\frac{z(t)}{t})&=(f')^{-1}(0)+\frac{1}{f''\Big((f')^{-1}(0)\Big)}\frac{z(t)}{t}+o(\frac{1}{t})\\
	&=\ovl{u}+\frac{p}{2f''(\ovl{u})t}+o(\frac{1}{t})
	\qquad \text{as}\ t\rightarrow +\infty.
	\end{aligned}
	\end{equation*}
	The estimates of $(f')^{-1}(\frac{z(t)-p}{t})$ is similar. So \eqref{lb11} is proved.
	
	\vspace{0.5cm}
	
	Combing Step 1-3, one can finish the proof of Theorem \ref{thmper}.
\end{proof}

%%%%%%%%%%%%%%%%%%%%%%%%%%%
% input a figure
%%%%%%%%%%%%%%%%%%%%%%%%%%%
%\begin{figure}[h!]
%\scalebox{.45}{\includegraphics{hexagonal_lattice.eps}}  
%\end{figure}
%%%%%%%%%%%%%%%%%%%%%%%%%%%

%%%%%%%%%%%%%%%%%%%%%%%%%%%%%%%%%%%%%%%%%%%%%%%%%%%%%%%%%%%%%%%%
%%%%% Bib
%%%%%%%%%%%%%%%%%%%%%%%%%%%%%%%%%%%%%%%%%%%%%%%%%%%%%%%%%%%%%%%%
\bibliographystyle{amsplain}
\bibliography{bib}

\end{document}